\documentclass[11pt,a4paper]{amsart}
\usepackage{amssymb,amsmath}
\usepackage{hyperref}
\usepackage{bm}
\usepackage[]{graphicx}
\usepackage{amssymb, epsfig}
\usepackage{latexsym}
\usepackage{amsmath}
\usepackage{amssymb}
\usepackage{amsfonts}
\usepackage{verbatim}
\usepackage{amscd}
\usepackage{mathrsfs}

\usepackage[latin1]{inputenc}
\usepackage{color}
\usepackage{hyperref}
\usepackage{bm}
\numberwithin{equation}{section}

\numberwithin{figure}{section}
\newtheorem{theorem}{Theorem}[section]
\newtheorem{lemma}[theorem]{Lemma}

\newtheorem{remark}[theorem]{Remark}

\begin{document}

\title[Stochastic Cahn-Hilliard/Allen-Cahn equation]{Existence and regularity of solution
for a\\  Stochastic Cahn-Hilliard/Allen-Cahn equation\\ with
unbounded noise diffusion}
\author[D. Antonopoulou]{Dimitra C.~Antonopoulou$^{\dag}$}
\thanks{$^{\dag}$ Department of Applied Mathematics, University of
Crete, and IACM, FORTH, Heraklion Greece}
\address{D. Antonopoulou, Department of Applied Mathematics, University of Crete,
GR--714 09 Heraklion, Greece, and, Institute of Applied and
Computational Mathematics, FO.R.T.H., GR--711 10 Heraklion,
Greece.}
\email{danton@tem.uoc.gr}
\author[G. Karali]{Georgia Karali$^{\dag}$}
\address{G. Karali, Department of Applied Mathematics, University of Crete,
GR--714 09 Heraklion, Greece, and, Institute of Applied and
Computational Mathematics, FO.R.T.H., GR--711 10 Heraklion,
Greece.}
\email{gkarali@tem.uoc.gr}
 \author[A. Millet]{Annie Millet$^{\ddag *}$}
 \thanks{$^{\ddag}$ SAMM (EA 4543), Universit\'e Paris 1 Panth\'eon Sorbonne {\it and}  $^{*}$ LPMA
(CNRS UMR 7599) }
\address
{A. Millet, SAMM (EA 4543), Universit\'e Paris 1 Panth\'eon Sorbonne,
 90 Rue de Tolbiac, 75634 Paris Cedex 13, France {\it  and }
Laboratoire de Probabilit\'es et Mod\`eles Al\'eatoires
(CNRS UMR 7599), Universit\'es Paris~6-Paris~7, Bo\^{\i}te Courrier 188,
4 place Jussieu, 75252 Paris Cedex 05,  France}
\email {annie.millet@univ-paris1.fr {\it and} annie.millet@upmc.fr}


\subjclass{35K55, 35K40, 60H30, 60H15.}

\begin{abstract}
The Cahn-Hilliard/Allen-Cahn equation with noise is a simplified
mean field model of stochastic microscopic dynamics associated
with adsorption and desorption-spin flip mechanisms in the
context of surface processes. For such an equation we consider a
multiplicative space-time white noise with diffusion coefficient
of sub-linear growth. Using technics from semigroup theory, we
prove existence, and path regularity of stochastic solution
depending on that of the initial condition. Our results are also
valid for the stochastic Cahn-Hilliard equation with unbounded
noise diffusion, for which previous results were established only
in the framework of a bounded diffusion
coefficient. We prove that the path regularity of stochastic
solution depends on that of the initial condition, and are identical to those proved for the stochastic
Cahn-Hilliard equation and a bounded noise diffusion coefficient. 
 If the initial condition vanishes, they are strictly less than $2-\frac{d}{2}$ in space and $\frac{1}{2}-\frac{d}{8}$ in time. 
As expected from the theory of parabolic operators in the sense
of Petrovsk\u{\i\i}, the bi-Laplacian operator seems to be
dominant in the combined model.
\end{abstract}

\maketitle \textbf{Keywords:} {\small{Stochastic
Cahn-Hilliard/Allen-Cahn equation, space-time white noise,
convolution semigroup, Galerkin approximations, unbounded
diffusion.}}

\section{Introduction}
\subsection{The Stochastic equation}
We consider the Cahn-Hilliard/Allen-Cahn equation with
multiplicative space-time noise:

\begin{equation} \left\{ \begin{array}{rll} \label{S-CH-AC}
u_t & =&-\varrho \Delta\Big(\Delta u-f(u)\Big)+\Big(\Delta
u-f(u)\Big)
+\sigma(u)\dot{W} \quad {\rm in}\quad {\mathcal{D}}\times [0,T),\\
u(x,0)& =&u_0(x)\quad {\rm in}\quad{\mathcal{D}},  \\ 
\frac{\partial u}{\partial \nu}&=&\frac{\partial \Delta u}{\partial
\nu}\; =\; 0 \quad {\rm on}\quad \partial{\mathcal{D}}\times [0,T).
\end{array} \right. 
\end{equation}
Here, ${\mathcal{D}}$ is a rectangular domain in $\mathbb{R}^d$
with $d=1,2,3$, $\varrho>0$ is a ``physical diffusion"   constant,
$f$ is a polynomial of degree 3 with a positive leading
coefficient, such as $f=F'$ where
 $F(u)=(1-u^2)^2$ is a double equal-well potential.
The ``noise diffusion"   coefficient $\sigma(\cdot)$ is a Lipschitz function
with sub-linear growth, $\dot W$ is a space-time
white noise in the sense of Walsh \cite{W}, and $\nu$ is the
outward normal vector. In addition, we assume that the initial
condition $u_0$ is sufficiently integrable or regular, depending on the desired results on the solution. Obviously,
when $\sigma:=1$, the noise in \eqref{S-CH-AC} becomes additive.

In this paper, as in \cite{CW}, we will analyze the more general
case of multiplicative noise. Thus, in the sequel we  will give
 sufficient conditions on the initial condition $u_0$ and on the coefficient
 $\sigma$ so that a unique
 global solution exists with Lipschitz path-regularity.
However, unlike \cite{CW}, we consider a more general
 Lipschitz coefficient $\sigma$ with sub-linear growth such that
\begin{equation}\label{undif}
|\sigma(u)|\leq C (1+|u|^\alpha),
\end{equation}
 for some $\alpha \in (0,
\frac{1}{9})$ and $C$ a positive constant.

The stochastic Cahn-Hilliard equation can be
considered as a special case of our model. Therefore, our method
shows that all the results of \cite{CW} on well-posedeness and
path-regularity for the solution of the stochastic Cahn-Hilliard
equation with a multiplicative noise extend to the case where the
function $\sigma$ satisfies the aforementioned sub-linear growth
assumption. Furthermore, using the factorization method for the
stochastic term we derive a better path regularity than that
obtained in \cite{CW}.

\subsection{The physical background}
Surface diffusion and adsorption/desorption consist the
micromechanisms that are typically involved in surface processes
or on cluster interface morphology. Chemical vapor deposition,
catalysis, and epitaxial growth are surface processes involving
transport and chemistry of precursors in a gas phase where the
unconsumed reactants and radicals adsorb onto the surface of a
substrate so that surface diffusion, or reaction and desorption
back to the gas phase is observed. Such processes have been
modelled by continuum-type reaction diffusion models where
interactions between particles are neglected or treated
phenomenologically, \cite{IE, E}. Alternatively, a more precise
microscocpic description is provided in statistical mechanics
theories, \cite{GB}. For instance we can consider a combination
of Arrhenius adsorption/desorption dynamics, Metropolis surface
diffusion and simple unimolecular reaction; the corresponding
mesoscopic  equation is:
\begin{equation}
u_t- D\nabla \cdot \Big[ \nabla u-\beta u(1-u)\nabla J*u\Big]-
\Big[ k_ap(1-u)-k_du\exp\big(-\beta J*u\big)\Big]+k_ru =0\, .
\label{mf3}
\end{equation}
Here,  $D$ is the diffusion constant, $k_r$, $k_d$ and $k_a$
denote respectively the reaction, desorption and adsorption
constants while $p$ is the  partial  pressure of the gaseous
species. The partial pressure $p$ is assumed to be a constant,
although realistically it is given by the fluids equations in the
gas phase. Furthermore, $J$ is the particle-particle interaction
energy and $\beta$ is the inverse temperature.

Stochastic microscopic dynamics such as Glauber and Metropolis
dynamics have been analyzed for adsorption/desorption-spin flip
mechanisms in the context of surface processes; for more details
we refer to the review article \cite{KV03}.
In addition, the Kawasaki and Metropolis stochastic dynamics
models describe the diffusion of a particle on a surface, where
sites cannot be occupied by more than one particle. Stochastic
time-dependent Ginzburg-Landau type equations with additive
Gaussian white noise source such as Cahn-Hilliard and Allen-Cahn
appear as Model B and Model A respectively in the classical
theory of phase transitions according to the universality
classification of Hohenberg and Halperin \cite{hohen}. A
simplified mean field mathematical model, associated with the
aforementioned mechanisms that describes surface diffusion,
particle-particle interactions and as well as adsorption to and
desorption from the surface, is a partial differential equation
written as a combination of Cahn-Hilliard and Allen-Cahn type
equations with noise. The Cahn-Hilliard operator is related to
mass conservative phase separation and surface diffusion in the
presence of interacting particles. On the other hand, the
Allen-Cahn operator is related to adsorption and desorption and
serves as a diffuse interface model for antiphase grain boundary
coarsening.

At large space-time scales the random fluctuations are suppressed
and a deterministic pattern emerges. Such a deterministic model
has been analyzed by Katsoulakis and Karali in \cite{KK}. The so
called mean field partial differential equation has the following
form:
\begin{equation}\label{uinitial}
\left\{ \begin{array}{rll} 
u_t&=&-\varepsilon^2\varrho\Delta\left(\Delta
u-\displaystyle\frac{f(u)}{\varepsilon^2}\right)
+\Delta u-\displaystyle\frac{f(u)}{\varepsilon^2},\\
u(x,0)&=&u_0(x),
\end{array} \right.
\end{equation}
where $f(u)=F^{\prime}$ for $F=(1-u^2)^2/4$ a double-well
potential with wells $\pm 1$, $\varrho>0$ is the diffusion
constant and $0<\varepsilon\ll1$ is a small parameter. In
\cite{KK}, the authors rigorously derived the macroscopic cluster
evolution laws and transport structure as a motion by mean
curvature depending on surface tension to observe that due to
multiple mechanisms an effective mobility speeds up the cluster
evolution.

\begin{remark}
The stochastic equation analyzed in this work is a simplified mean
field model for interacting particle systems used in statistical
mechanics. These systems are Markov processes set on a lattice
corresponding to a solid surface.  A  typical example is the
Ising-type systems defined on a multi-dimensional lattice; see 
\cite{GLP}. Assuming that the particle-particle interactions are
attractive, then the resulting system's Hamiltonian is
nonnegative (attractive potential). Hence, the diffusion constant
$\varrho$ of the SPDE \eqref{S-CH-AC} is considered positive, as in 
\cite{KK}.
\end{remark}
\begin{remark}
Ginzburg-Landau type operators are usually supplemented by
Neumann or periodic boundary conditions. In order to obtain an
initial and boundary value problem we consider the SPDE
\eqref{S-CH-AC} with the standard homogeneous Neumann boundary
conditions on $u$ and $\Delta u$. These conditions are frequently
used for the deterministic or stochastic Cahn-Hilliard equation; see e.g. 
\cite{El-Z,DaDe,CW}.
\end{remark}

\subsection{Main results}
As a first step for a rigorous mathematical analysis of the
stochastic model, in Section \ref{sec2}, we will prove existence
and uniqueness of a solution to
\eqref{S-CH-AC}  
 when the initial condition
$u_0$ belongs to $L^q({\mathcal{D}})$ for $q\in[3,\infty)$ if
$d=1,2$ and $q\in [6,\infty)$ if $d=3$. Section \ref{gener}
describes some possible general assumptions on the domain
${\mathcal{D}}$ which would lead to the same result and presents
the stochastic Cahn-Hilliard equation as a special case of a
Cahn-Hilliard/Allen-Cahn stochastic model.   Note that the approach used in this paper to solve
this non linear SPDE with a polynomial growth is similar to that developed  by  J.B. Walsh \cite{W} and I. Gy\"ongy
\cite{Gy} for the stochastic heat equation and related SPDEs.  Note that unlike these references, 
the smoothing effect of the bi-Laplace operator enables us to deal with a stochastic
perturbation driven by a space-time white noise in dimension 1 up to 3. 

The existence-uniqueness proof is similar to that of Cardon-Weber
in \cite{CW}, and relies on upper estimates of the fundamental
solution, Galerkin approximations and the application of a
cut-off function. However, the fact that the diffusion
coefficient $\sigma$ is unbounded requires to multiply $\sigma$
by the cut-off function in order to estimate properly the
stochastic integral, and then to use \textit{a priori} estimates for
the remaining part.

With our method we prove existence of a global solution under the
requirement that $\sigma$ satisfies the following sub-linear
growth condition: $|\sigma(u)|\leq C(1+|u|^\alpha)$ for some
$\alpha \in (0,\frac{1}{9})$ and some positive constant $C$. 
Finally note that the upper estimates on the Green's function stated in section \ref{sec2} obviously show that
all the results in \cite{CW} can be extended to our framework if $\sigma$ is bounded. Therefore, 
one of the main contributions of this paper is to deal with some unbounded noise coefficient $\sigma$.
Note that we could not apply the technique introduced by S.~Cerrai
\cite{Cerrai} for the stochastic Allen-Cahn equation (see also
the work of
 M.~Kuntze and J. van Neerven, \cite{Kuntze-vanNeernen}, for a more general framework) and obtain global
solutions for more general diffusion coefficients. This is due to
the fact that in our model, in contrast to the Allen-Cahn
equation, the Laplace operator is applied to the nonlinearity.
However, we believe that global solutions could exist for 
unbounded noise diffusion coefficients satisfying a growth condition  more general  than
\eqref{undif}.

Path regularity of the solution is proved in section \ref{path}.
If the initial condition vanishes, the domain ${\mathcal{D}}$ can
again be quite general. Otherwise, we have to impose that
${\mathcal{D}}$ is a rectangle.

Our method shows that all the results of \cite{CW} on
well-posedeness and path regularity of the solution to the
stochastic Cahn-Hilliard equation extend to the case of an unbounded
noise diffusion.  
In addition, we prove path regularity of  the considered Cahn-Hilliard/Allen Cahn
SPDE. The method, based on the factorization method for the deterministic and random forcing
terms,  yields the same regularity as  that proven in \cite{CW}, where $\sigma$ is bounded.

As usual we demote by $C$ a generic constant and by $C(s)$ a constant depending 
on some parameter $s$. For $p\in [1,\infty]$, the $L^p({\mathcal{D}})$-norm is denoted by  $\|\cdot\|_{p}$.
Finally, given real numbers $a$ and $b$ we let $a\vee b$ (resp. $a\wedge b$) denote the maximum  (resp. the minimum) 
of $a$  and $b$.

\section{Existence of stochastic solution}\label{sec2}
\subsection{Preliminaries}
 For simplicity and to ease notation, without restriction of generality, we will assume that 
 the ``physical diffusion" constant 
$\varrho$ is equal to 1 and that ${\mathcal{D}}$ is the unitary
cub. Extension to more general domains will be addressed in the
next section.

In order to give a mathematical meaning to  the stochastic PDE
\eqref{S-CH-AC} we integrate in time and space and use the
initial and boundary conditions (see e.g. \cite{W}). For a strict
definition of solution, we say that $u$ is a weak (analytic)
solution of the equation \eqref{S-CH-AC} if it satisfies the
following weak formulation:
\begin{align}\label{S-CH-AC-W}
&\int_{\mathcal{D}}\Big(u(x,t)-u_0(x)\Big)\phi(x)\; dx=\nonumber\\
&\int_0^t \int_{\mathcal{D}}\Big(-
\Delta^2 \phi(x)u(x,s)
+\Delta\phi(x)[
f(u(x,s))+u(x,s)]
-\phi(x) f(u(x,s)\Big)\; dxds \nonumber \\
&+\int_0^t\int_{\mathcal{D}}\phi(x)\sigma(u(x,s))\; W(dx, ds),
\end{align}
for all $\phi\in {\mathcal C}^4({\mathcal{D}})$ with
$\frac{\partial\phi}{\partial
\nu}=\frac{\partial\Delta\phi}{\partial \nu}=0$ on $\partial
{\mathcal{D}}.$ Note that this $u$ stands as a probabilistic
"strong solution" since we keep the given space-time white noise
and do not only deal with the distribution of the processes.

The random measure $W(dx,ds)$ is the $d$-dimensional space-time
white noise, that is induced by the one-dimensional
$(d+1)$-parameter (with $d$  space variables and one time
variable) Wiener process $W$ defined as $W:=\big\{ W(x,t):\;
t\in[0,T],\; x\in{\mathcal{D}}\big\}$. For every $t\geq 0$ we
let  ${\mathcal F}_t := \sigma \big( W(x,s):\; s\leq t,\;
x\in {\mathcal{D}}\big)$ denote  the filtration generated by $W$, cf.
\cite{W,CW,AK}. Furthermore, we assume that the coefficient
$\sigma : {\mathbb R} \to {\mathbb R}$ is a Lipschitz function and
satisfies the following growth condition for some $\alpha \in (0,
1]$ and $C>0$:
\begin{equation*}
|\sigma(x)|\leq C (1+|x|^\alpha), \; \forall x\in {\mathbb R}.
\end{equation*}

\subsection{Some results on the Green's function}\label{secGreen}
Let $\Delta$ denote the Laplace operator; we shall use the
Green's function for the operator ${\mathcal
T}:=-\Delta^2+\Delta$ on ${\mathcal{D}}$ with the homogeneous
Neumann conditions, that is the fundamental solution to
$\partial_t u - {\mathcal T}u =0$ on ${\mathcal{D}}$ with the
boundary conditions 
$ \frac{\partial u}{\partial \nu}=\frac{\partial \Delta u}{\partial
\nu}=0$  on $ \partial{\mathcal{D}}\times [0,T)$.
Let $k=(k_i,\;i=1,\cdots, d)$ denote a multi-index with
non-negative integer components $k_i$ and let
${ \|k\|^2 }:=\displaystyle{\sum_i} k_i^2$. We set
$\epsilon_0(x):=\frac{1}{\sqrt{\pi}}$, and for any positive
integer $j$ we define $\epsilon_j(x):=\sqrt{\frac{2}{\pi}}
\cos(jx)$. Finally for $k=(k_i) \in {\mathbb N}^d$ and $x\in
{\mathcal{D}}$ let $\epsilon_k(x):= \displaystyle{\prod_i}
\epsilon_i(x_i)$. Then $(\epsilon_k,\;k\in {\mathbb N}^d)$ is an
 orthonormal basis of
$L^2({\mathcal{D}})$ consisting on eigenfunctions of $ {\mathcal
T}$ corresponding to  the  eigenvalues $-\lambda_k^2 - \lambda_k$ where
{ $\lambda_k=\|k\|^2$. } Of course, $\epsilon_0$ is related to the
null eigenvalue.

Let $S(t):=e^{(-\Delta^2+\Delta)t}$ be the semi-group generated
by the operator ${\mathcal T}$;   if $u:=\sum_{k}(u,\epsilon_k) \,
\epsilon_k$ then
$${\mathcal T} u=\sum_{k} -(\lambda_k^2 + \lambda_k) (u,\epsilon_k)_{L^2(\mathcal{D})} \, \epsilon_k,$$
and (see e.g.  \cite{DaDe,CW}) the
convolution semigroup is defined by
$$S(t)U(x):= \sum_{k} e^{-(\lambda_k^2+\lambda_k)
t}(U,\epsilon_k)_{L^2(\mathcal{D})}\epsilon_k(x),$$ for any $U$ in
$L^2({\mathcal{D}})$ with the associated Green's function given by
\begin{equation}\label{Green}
G(x,y,t)=\sum_{k}
e^{-(\lambda_k^2 + \lambda_k) t}\, \epsilon_k(x)\, \epsilon_k(y)
\end{equation}
for $t>0$, $x,y\in {\mathcal D}$.
Using the Definition 1.3 of \cite{EZ}, we deduce that ${\mathcal
T}=-\Delta^2+\Delta$ is uniformly strongly parabolic in the sense
of Petrovsk\u{\i\i}. Thus, as proved in \cite{EI}, the following
upper estimates of the Green function $G$ and its various
derivatives hold true. Notice that they are similar to those of
the Green's function used in \cite{CW} for the operator
$-\Delta^2$.
\begin{lemma}\label{G-est-lemma}
Let $G$ be the Green's function defined by \eqref{Green}. Then
there exist positive constants $c_1$ and $c_2$ such that for any
$t\in(0,T]$,  any $x,y\in{\mathcal{D}}$ and any multiindex $k=
(k_i,\; i=1,\cdots,d)$ with
 $|k|=\sum_{i=1}^d k_i \in \{1,2\}$, the next inequalities are satisfied:
\begin{eqnarray}\label{G-est-1}
|G(x,y,t)|&\leq&
c_1\, t^{-\frac{d}{4}}\exp\Big(-c_2\, {|x-y|^{\frac{4}{3}}}\, {t^{-\frac13}}\Big),\\
\label{G-est-2}
|\partial^k_x G(x,y,t-s)|&\leq&
c_1\, t^{-\frac{d+|k|}{4}}\exp\Big(-c_2\, {|x-y|^{\frac{4}{3}}}\, t^{-\frac13}\Big),\\
\label{G-est-3}
|\partial_t G(x,y,t-s)|&\leq&
c_1\, t^{-\frac{d+4}{4}}\exp\Big(-c_2\, {|x-y|^{\frac{4}{3}}}\, {t^{-\frac13}}\Big).
\end{eqnarray}
\end{lemma}

Furthermore, given any $c>0$ there exists a positive constant
$C(c)$ such that
\begin{equation} \label{expexp}
\int_{{\mathbb R}^d} \exp\big(-c|x|^{\frac{4}{3}} t^{-\frac{1}{3}}\big) dx  = C(c) t^{\frac{d}{4}}.
\end{equation}
For $x\in {\mathcal D}$, $t > s\geq 0$, set 
\begin{equation} \label{defh}
h(x,t,s)=-c_2\,  |x|^{\frac{4}{3}}\, (t-s)^{-\frac{1}{3}}.
\end{equation} 
 The following lemma gathers several estimates for integrals of space (respectively time) increments
of $G$. Note once more that the results are the same as those of Lemma
1.8 in \cite{CW} and are deduced from the explicit formulation \eqref{Green} of $G$ by using similar
arguments.
\begin{lemma}\label{G-holder}
Let $G$ be the Green's function defined by \eqref{Green}. Given
positive constants $\gamma, \gamma '$ with $\gamma<(4-d)$,
$\gamma\leq 2$ and  $\gamma '<1-\frac{d}{4}$, there exists a
constant $C>0$
  such that for any $t>s\geq 0$ and any $x,y \in {\mathcal{D}}$ the
  next estimates hold true:
\begin{eqnarray}\label{G-est-4}
\int_0^t\int_{\mathcal{D}}|G(x,z,t-r)-G(y,z,t-r)|^2\;dzdr
&\leq& C|x-y|^\gamma,\\
\label{G-est-5}
\int_0^s\int_{\mathcal{D}}|G(x,z,t-r)-G(x,z,s-r)|^2\;dzdr& \leq
&  C|t-s|^{\gamma '},\\
\label{G-est-6} \int_s^t\int_{{\mathcal{D}}}|G(x,z,t-r)|^2\;dzdr
&\leq& C |t-s|^{\gamma '}.
\end{eqnarray}
\end{lemma}

%
\subsection{Integral representation}
Using the Green's function, we can present the solution of
equation \eqref{S-CH-AC-W} in an integral form for any $x\in
{\mathcal{D}}$ and $t\in[0,T]$, that is the following mild solution:
\begin{align}\label{intf}
 u(x,t)
=&\int_{{\mathcal{D}}}u_0(y)G(x,y,t)\;dy  \nonumber \\
&+\int_0^t \int_{{\mathcal{D}}}\big[ \Delta G(x,y,t-s) -
G(x,y,t-s)\big] f(u(y,s))\;dyds
\nonumber \\
& +\int_0^t\int_{{\mathcal{D}}}G(x,y,t-s)\sigma(u(y,s))\; W(dy, ds). 
\end{align}

Application of the inequality \eqref{expexp} and H\"older's
inequality lead to the following bound for the term involving
the initial condition. 
\begin{lemma}
Let $G(x,y, t)$ be the  Green's function defined by \eqref{Green}.
For every $1\leq q<\infty$ and $T>0$ there exists a constant
$C:=C(T,q) $ such that
\begin{equation}\label{evolution-initial-value}
\sup_{t\in[0,T]}\|G_t u_0\|_{q}\leq C \|u_0\|_q,
\end{equation}
where $G_0=Id$ and  $G_t u_0$ is defined for $t>0$ by
\begin{equation}\label{Gt}
G_t u_0(x):=\int_{\mathcal{D}}u_0(y)G(x,y,t)\; dy .
\end{equation}
\end{lemma}
\subsection{Truncated equation}\label{sectruncated}
In order to prove the existence of the solution $u$ to \eqref{intf},   as a first
step we consider an appropriated cut-off SPDE, cf. \cite{CW}. Let
$\chi_n\in C^1(\mathbb{R} , \mathbb{R}^+)$ be a cut-off function
satisfying $|\chi_n|\leq 1$,   $|\chi'_n|\leq 2$ for any $n>0$ and
$$  \chi_n(x) =\begin{cases}
    1 & \text{if} \;\;\; |x| \leq n, \\
    0 & \text{if} \;\;\; |x| \geq n+1.
  \end{cases}
$$

For fixed $n>0$, $x\in {\mathcal{D}}$,  $t\in[0,T]$ and $q\in [3,+\infty)$,  we consider
the following cut-off SPDE:
\begin{align}\label{u_n}
u_n(x,t)
=&\int_{{\mathcal{D}}}u_0(y)G(x,y,t)\;dy \nonumber \\
&+\int_0^t \int_{{\mathcal{D}}} \big[ \Delta G(x,y,t-s) - G(x,y,t-s)\big]\, \chi_n(\|u_n(\cdot,s)\|_q)\, f(u_n(y,s))\, dyds \nonumber \\
&+\int_0^t\int_{{\mathcal{D}}}G(x,y,t-s) \,
\chi_n(\|u_n(\cdot,s)\|_q)\,\sigma(u_n(y,s))\; W(dy, ds).
\end{align}

In this section we suppose that  $\sigma$ satisfies \eqref{undif} with $\alpha\in (0,1]$, 
and that the following condition {\bf (C$_\alpha$)} holds:
 \\
 {\bf Condition (C$_\alpha$)} {\it One of the following properties (i) or (ii) is satisfied: \\
 \indent (i) $d=1,2$ and $q\in [3, +\infty)$,  or $d=3$ and $q\in [6,+\infty)$,\\
\indent (ii)  $d=3$ and $q\in \big(3\vee [6(1-\alpha)],6\big)$. }
\\

 We show the existence and uniqueness of the solution to
the SPDE \eqref{u_n} in the set $\mathcal{H}_T$ defined by
\[   
\mathcal{H}_T:=\Big\{  u(\cdot, t)\in L^q({\mathcal{D}})\;
\mbox{\rm for } t\in [0,T]: \; u \; {\rm
is}\: (\mathcal{F}_t) \mbox{\rm -adapted}    
\; {\rm and }\;  \| u\|_{\mathcal{H}_T}<{\infty} \Big\},
\] 
where
\begin{equation}
\|u\|_{\mathcal{H}_T}:=\sup_{t\in[0,T]} E\Big{(}\|u(\cdot
,t)\|_q^{\beta}\Big{)}^{\frac{1}{\beta}},
\end{equation}
for  $\beta\in(\frac{q}{\alpha},\infty)$ if Condition {\bf
(C$_\alpha$)(i)} holds, or  for $\beta\in (\frac{q}{\alpha},
\frac{6q}{(6-q)})$ if Condition {\bf(C$_\alpha$)(ii)} holds.

\begin{remark}
(i) In order to present our results in a more general framework we
consider the growth condition \eqref{undif} with $\alpha\in
(0,1]$; the upper bound of $\alpha$ will be restricted in the
sequel.\\
(ii) Note that if $d=3$, the inequality $6(1-\alpha)<q<6 $ implies that the interval
$ (\frac{q}{\alpha}, \frac{6q}{(6-q)})$  is not empty. 
\end{remark}
\begin{theorem}
Let $\sigma$ be globally Lipschitz and satisfy the assumption
\eqref{undif} with $\alpha \in (0,1]$, let $u_0\in L^q({\mathcal
D})$ and let Condition {\bf (C$_\alpha$)} hold. Furthermore, let $\beta \in
(\frac{q}{\alpha},+\infty)$ if Condition {\bf (C$_\alpha$)(i)} is
satisfied (resp. $\beta \in (\frac{q}{\alpha},\frac{6q}{6-q})$ if
Condition {\bf (C$_\alpha$)(ii)} is satisfied). Then the SPDE \eqref{u_n}
admits a unique solution $u_n$ in every time interval $[0,T]$  
and $u_n\in {\mathcal H}_T$.
\end{theorem}
\begin{proof}
We define the operators $\mathcal{M}$ 
and $\mathcal{L}$ on $\mathcal{H}_T$ by
\begin{align}
\mathcal{M}(u)(x,t) : = &
\int_0^t \!\! \int_{{\mathcal{D}}}[ \Delta G(x,y,t-s) - G(x,y,t-s)] \chi_n(\|u(\cdot,s)\|_q) f(u(y,s)) \,dyds ,   \label{M}\\
\mathcal{L}(u)(x,t) :=&  \int_0^t \!\!
\int_{{\mathcal{D}}}G(x,y,t-s) \,
\chi_n(\|u(\cdot,s)\|_q)\,\sigma(u(y,s))\; W(dy, ds), \label{L}
\end{align}
with $u\in\mathcal{H}_T$. Then obviously \eqref{u_n} is written as
\begin{equation}\label{u_nn}
u_n(x,t)= \int_{{\mathcal{D}}}u_0(y)G(x,y,t)\;dy +{\mathcal
M}(u_n)(x,t)+{\mathcal L}(u_n)(x,t).
\end{equation}
We claim that if $T>0$ is sufficiently small, then the operator
$\mathcal{M}+\mathcal{L}$ is a contraction
mapping from $\mathcal{H}_T$ to $\mathcal{H}_T$.

First we consider the  mapping $\mathcal{M}$. For an arbitrary function
$u\in\mathcal{H}_T$, by Minkowski's inequality, \eqref{G-est-1}  and \eqref{G-est-2}
we have
\begin{align*}
\|\mathcal{M}(u)& (\cdot,t)\|_q\leq  c_1\int_0^t
(t-s)^{-\frac{d+2}{4}} \\
&\times \Big{\{} \int_{{\mathcal{D}}} \Big{|}\int_{{\mathcal{D}}}
\exp\Big{(}-c_2
\frac{|x-y|^\frac{4}{3}}{(t-s)^{\frac{1}{3}}}\Big{)}
\chi_n(\|u(\cdot,s)\|_q) f(u(y,s)) \; dy\Big{|}^q \;
dx\Big{\}}^{\frac{1}{q}}\; ds.
\end{align*}
By using Young's inequality with exponents $\rho$ and $r$ in
$[1,\infty)$ such that
$\frac{1}{\rho}+\frac{1}{r}=\frac{1}{q}+1$,   we obtain for $h(x,t,s):=- c_2
\frac{|x|^\frac{4}{3}}{(t-s)^{\frac{1}{3}}}$ defined by \eqref{defh} 
\begin{align} \label{m_1-1-5} 
\|\mathcal{M}(u)(\cdot,t)\|_q & \leq c_1 \int_0^t \!
(t-s)^{-\frac{d+2}{4}}\|\exp(h(\cdot,t,s))\|_{r}
\Big{\|}\chi_n(\|u(\cdot,s)\|_q)  f(u(.,s))\Big{\|}_{\rho} \; ds \nonumber \\
& \leq C \int_0^t
(t-s)^{-\frac{d+2}{4}+\frac{d}{4r}}
\Big{\|}\chi_n(\|u(\cdot,s)\|_q)f(u(.,s))\Big{\|}_{\rho} \; ds,
\end{align} 
where the last inequality follows from \eqref{expexp}.
We choose $\rho = \frac{q}{3}\geq 1$ since $q\in [3,\infty)$  and $r \in [1,\infty)$ satisfying
$\frac{1}{\rho}+\frac{1}{r}=\frac{1}{q}+1$ . %
The function  $f$ is a polynomial of degree 3, so, for  $n\geq 1$
we have
\begin{equation}
\Big{\|} \chi_n(\| u(\cdot, s) \|_{q}) f(u(.,s))
\Big{\|}_{\frac{q}{3}} \leq  C n^3.
\end{equation}
Since $q>d$ we deduce that  $-\frac{d+2}{4}+\frac{d}{4r}  > -1$; hence the above inequalities
yield
\begin{equation}\label{last-est-m_1}
\|\mathcal{M}(u)\|_{\mathcal{H}_T} =\sup_{t\in[0,T]}E\Big{(}\|
\mathcal{M}(u(\cdot,t)) \|_q^\beta\Big{)}^{\frac{1}{\beta}}\leq
C \; n^3\;  T^{-\frac{d}{4}+\frac{d}{4r}+\frac{1}{2}}.
\end{equation}
Therefore, $\mathcal{M}$ is a mapping from $\mathcal{H}_T$ to
$\mathcal{H}_T.$ Moreover,  for arbitrary $u$ and $v$ in
$\mathcal{H}_T$ such that $\|u(\cdot ,s)\|_q\leq\|v(\cdot,
s)\|_q$,  we shall prove that for  $q\in [3,\infty)$ and  $\rho = \frac{q}{3}$ the
next inequality holds true:
\begin{equation}\label{claim_m_1}
\Big{\|}\chi_n(\|u(\cdot,s)\|_q)f(u(\cdot,s)) -
\chi_n(\|v(\cdot,s)\|_q)f(v(\cdot,s)) \Big{\|}_{\rho} \leq C\, n^3 \|
u(\cdot,s)-v(\cdot,s) \|_{q}.
\end{equation}
Indeed, we have
\begin{align*}
\Big{\|}\chi_n(&\|u(\cdot,s)\|_q)    f(u(\cdot,s)) -
\chi_n(\|v(\cdot,s)\|_q)f(v(\cdot,s)) \Big{\|}_{\rho} \\
& \leq \Big{\|} \big[ \chi_n(\|u(\cdot,s)\|_q) -
\chi_n(\|v(\cdot,s)\|_q)\big] \, f(u(\cdot,s))
\Big{\|}_{\rho} 
+\Big{\|}\chi_n(\|v(\cdot,s)\|_q)\, \big[ f(u(\cdot,s)) -
f(v(\cdot,s)) \big]  \Big{\|}_{\rho}.
\end{align*}
Note that  $\|v(\cdot,s)\|_q\geq \|u(\cdot,s)\|_q$ and
$$\chi_{n}(\|u(\cdot,s)\|_{q})-\chi_n(\|v(\cdot,s)\|_q)=0\;\;\mbox{if}\;\;
\|u(\cdot,s)\|_q\geq n+1.$$ 
Hence, for $n\geq 1$ and
$\rho=\frac{q}{3}$ we obtain the existence of $C>0$ such that for all $n\geq 1$
\begin{align*}
\Big{\|} \big[ \chi_n(\|u(\cdot,s)\|_q) -
\chi_n(\|v(\cdot,s)\|_q)\big] \, f(u(\cdot,s))
\Big{\|}_{\rho}&\leq
 C \Big{(}1+(n+1)^3\Big{)}\Big{|}\|u(\cdot,s)\|_q-\|v(\cdot,s)\|_q\Big{|}\\
&\leq  C\; n^3\; \|u(\cdot,s)-v(\cdot,s)\|_q.
\end{align*}
Using again the inequality $\|v(\cdot,s)\|_q\geq
\|u(\cdot,s)\|_q$, then for any $n\geq 1$ we deduce the existence of $C>0$ such that 
\begin{align*}
\Big{\|}\chi_n(\|v(\cdot,s)\|_q) & \, \big[ f(u(\cdot,s)) -
f(v(\cdot,s)) \big]  \Big{\|}_{\rho} \\
&\leq C
\chi_n(\|v(\cdot,s)\|_q)\Big(1+\|v(\cdot,s)\|_q^2+\|u(\cdot,s)\|_q^2\Big)\|u(\cdot,s)-v(\cdot,s)\|_q\\
&\leq C n^2\, \|u(\cdot,s)-v(\cdot,s)\|_q.
\end{align*}
holds for any $n\geq 1$.  Thus, \eqref{claim_m_1} holds true.

Inequality \eqref{claim_m_1} and an argument similar to that used
for proving \eqref{m_1-1-5} yield
\begin{align}\label{M_1-1}
\|\mathcal{M}&(u) (\cdot,t)-\mathcal{M}(v)(\cdot,t)\|_q \nonumber \\
 &\leq  \int_0^t |t-s|^{-\frac{d+2}{4}+\frac{d}{4r}}\Big{\|}
  \chi_n(\|u(\cdot,s)\|_{q})f(u(\cdot,s))-\chi_n(\|v(\cdot,s)\|_{q})f(v(\cdot,s)) \Big{\|}_{\rho}\; ds\nonumber \\
 &\leq C\,  n^3 \int_0^t |t-s|^{-\frac{d+2}{4}+\frac{d}{4r}}  \| u(\cdot,s)-v(\cdot,s)\|_q\; ds.
\end{align}
Therefore, by inequality \eqref{M_1-1} and H\"older's inequality,
since $\beta \in [q,\infty)$, we deduce
\begin{align}
 \|\mathcal{M}(u)-\mathcal{M}(v)\|_{\mathcal{H}_T}
  &
 \leq C \,  n^{3} \sup_{t\in [0,T]}\Big\{  E \Big{|}
 \int_0^t |t-s|^{-\frac{d+2}{4}+\frac{d}{4r}}  \| u(\cdot,s)-v(\cdot,s)\|_q\; ds\Big{|}^{\beta} \Big\}^{1/\beta}
\nonumber \\
 &
 \leq
 C\, n^{3}\, T^{(-\frac{d+2}{4}+\frac{d}{4r}+1)}\sup_{t\in[0,T]}E\Big{(}\|u(\cdot,t)-v(\cdot,t)\|_{q}\Big{)} \nonumber \\
& \leq C \, n^3 \, T^{-\frac{d+2}{4}+\frac{d}{4r}+1}\|u-v\|_{\mathcal{H}_T}. \label{M_1-2}
\end{align}
Obviously, by \eqref{last-est-m_1} and \eqref{M_1-2} it follows
that for fixed $n\geq 1 $  and $T>0$, the map $\mathcal{M}$ is  Lipschitz   from $\mathcal{H}_T$ to
$\mathcal{H}_T$.

For the mapping ${\mathcal L}$ defined in terms of a stochastic integral,    at first notice
that since $\alpha \in (0,1]$, the inequality $\beta > \frac{q}{\alpha}$ yields $\beta \in (q,\infty)$.  Thus the H\"older, Burkholder and
Minkowski  inequalities, and the growth condition \eqref{undif} on $\sigma$ yield 
\begin{align*}
E\|{\mathcal L}(u(\cdot,t))\|_q^\beta & \leq C \int_{\mathcal{D}} E |{\mathcal L}(u(x,t)|^\beta \,dx \\
& \leq C \int_{\mathcal{D}} E \Big| \int_0^t \int_{\mathcal{D}}
\big| G(x,y,t-s) \, \chi_n(\|u(\cdot,s)\|_q)\,  \sigma(u(y,s))|^2
\,
dy ds \Big|^{\beta/2} dx \\
&\leq C  \Big( E \int_0^t \Big\| \int_{\mathcal{D}}
 G^2(\cdot,y,t-s)\,  \chi_n(\|u(\cdot,s)\|_q) \, \big[ 1+|u(y,s)|^{2\alpha}\big]\,  dy \Big\|_{\beta/2} ds \Big)^{\beta/2}.
\end{align*}
 Since $\beta \in (\frac{q}{\alpha},\infty)$, we have $\frac{2\alpha}{q} > \frac{2}{\beta}$ and  we may choose $\bar{r}\in (1,\infty)$
such that $\frac{2\alpha}{q}+\frac{1}{\bar{r}} =
\frac{2}{\beta}+1$. 
 Let once more $h(x,t,s)$ be defined by \eqref{defh}; 
Young's inequality and \eqref{G-est-1} imply
\begin{align*}
E\|{\mathcal L}(u(\cdot,t)\|_q^\beta \leq & C \Big( E\int_0^t
(t-s)^{-\frac{d}{2}} \Big\| \exp( h(\cdot,t,s)) \Big\|_{\bar{r}}
\chi_n(\|u(\cdot,s)\|_q) \, \big\| [1+|u(\cdot,s)|^{2\alpha} \big]
 \big\|_{\frac{q}{2\alpha}} ds \Big)^{\beta/2}\\
\leq  & C \Big( E\int_0^t (t-s)^{-\frac{d}{2} + \frac{d}{4\bar{r}}} (1+n^{2\alpha}) ds\Big)^{\beta/2}.
\end{align*}
Note that  the inequalities $d<4$, $q\geq 3$, $\alpha \in (0,1]$  and $\beta > \frac{q}{\alpha}$ yield  
 $-\frac{d}{2}+\frac{d}{4\bar{r}} > -1$. 
 Hence, for any $u\in {\mathcal H}_T$ we obtain the existence of $C>0$ such that  
\begin{equation}\label{LLL}
\| {\mathcal L}(u)\|_{{\mathcal H}_T} \leq C (1+n^\alpha) T^{\frac{1}{2} [ -\frac{d}{2}+\frac{d}{4\bar r} +1]}
\end{equation}
holds for every $n\geq 1$,   and therefore, $\mathcal{L}$ is also a mapping from
$\mathcal{H}_T$ to $\mathcal{H}_T$. Recall that $\sigma$ is
Lipschitz. Therefore,  an argument similar to that used to prove
\eqref{claim_m_1} with $q$ instead of $\rho$ shows that for
$u,v\in {\mathcal H}_T$, we have
\begin{equation} \label{claim_L}
\| \delta(u,v,\cdot,s)\|_q \leq C (1+n^\alpha)
\|u(.,s)-v(\cdot,s)\|_q,
\end{equation}
for
$$\delta(u,v,y,s):= \chi_n(\|u(\cdot,s)\|_q) \sigma(u(y,s)) -
\chi_n(\|v(\cdot,s)\|_q) \sigma(v(y,s)).$$
Recall that $\alpha \in (0,1]$ and $\beta > \frac{q}{\alpha}$, so that $\beta >q$; thus  the H\"older, Burkholder-Davies-Gundy
and Minkowski inequalities together with \eqref{G-est-1} yield for
$u,$ $v$ in $\mathcal{H}_T$
\begin{align*}
E\|  \mathcal{L}(u) (\cdot,s)-&\mathcal{L}(v)(\cdot,s)\|_q^\beta
 \leq
C \int_{{\mathcal{D}}}E|\mathcal{L}(u)(\cdot,s)-\mathcal{L}(v)(\cdot,s)|^\beta\; dx \\
& \leq C \int_{\mathcal{D}} \Big| \int_0^t \!\! \int_{\mathcal{D}}
G^2(x,y,t-s)
\big| \delta(u,v,y,s)\big|^2 dyds\Big|^{\beta /2} dx \\
&\leq  C E \Big| \int_0^t (t-s)^{-\frac{d}{2}} \Big\| \exp(
h(\cdot,t,s)) * \delta^2(u,v,\cdot,s) \Big\|_{\beta/2} ds
\Big|^{\beta/2}.
\end{align*}
The inequality $\beta >q$ implies the existence of $r_2\in
(1,\infty)$ such that $\frac{2}{\beta} + 1 = \frac{2}{q} +
\frac{1}{r_2}$. Using once more the
assumptions on $q,\alpha$ and $\beta$ in Condition (C$_\alpha$), 
 in particular the assumption $\beta (6-q)<6q$ for $d=3$ and $q\in [3,6)$,  
 we deduce $-\frac{d}{2} + \frac{d}{4 r_2}>-1$.   Thus Young's inequality and 
\eqref{claim_L} imply 
\begin{align*}
E\| {\mathcal L}(u)(\cdot,s)& - {\mathcal L}(v)(\cdot,s)\|_q^\beta
\leq C E \Big| \int_0^t (t-s)^{-\frac{d}{2} + \frac{d}{4 r_2}}
\| \delta(u,v,\cdot,s)\|_{q/2} ds \Big|^{\beta/2}\\
&\leq  C (1+n^{\alpha \beta}) T^{(-\frac{d}{2} + \frac{d}{4 r_2}
+1)\frac{\beta}{2}} \sup\Big{\{}
E\|u(\cdot,s)-v(\cdot,s)\|_q^\beta : t\in [0,T]\Big{\}},
\end{align*}
and therefore,
\begin{equation}\label{L-1}
 \|\mathcal{L}(u)-\mathcal{L}(v)\|_{\mathcal{H}_T}
 \leq C\,(1+n^\alpha)\,  T^{-\frac{d}{4}+\frac{d}{8r_2}+\frac{1}{2}}\, \|u-v\|_{\mathcal{H}_T}.
\end{equation}
So, for fixed $n$ and $T>0$, the map  $\mathcal{L}$ is also a
Lipschitz mapping from $\mathcal{H}_T$ to $\mathcal{H}_T$.

The upper estimates \eqref{M_1-2} and \eqref{L-1}
 imply that the mapping $\mathcal{M}+\mathcal{L}$ is Lipschitz from $\mathcal{H}_T$
to $\mathcal{H}_T$ with the Lipschitz constant bounded by
$$C(n,T):= C \Big{[} n^3 T^{-\frac{d+2}{4}+\frac{d}{4r}+1} + C
n^\alpha T^{-\frac{d}{4}+\frac{d}{8r_2}+\frac{1}{2}}\Big{]}.$$

For fixed $n\geq 1$,  there exists $T_0(n)$  sufficiently small
(which does not depend on $u_0$) such that $C(n,T)<1$ for $T\leq
T_0(n)$, so that $\mathcal{M}+\mathcal{L}$ is a contraction
mapping from the space ${\mathcal H}_T$ into itself.
 Thus for $T\leq T_0(n)$,  the map ${\mathcal M}+{\mathcal L}$ has a unique fixed point in the set $\Big{\{}
u\in\mathcal{H}_T:\; u(\cdot,0)=u_0 \Big{\}}$. This implies that
in $[0,T]$, for $T\leq T_0(n)$, there exists a unique solution
$u_n$ for the SPDE \eqref{u_n}.

If $T>T_0(n)$, let $\bar{u}_0(x)=u_n(x, T_0(n))$ and
$\bar{W}(t,x)=W(T_0(n)+t,x)$; then $\dot{\bar{W}}$ is a space-time
white noise related to the filtration $({\mathcal
F}_{T_0(n)+t},t\geq 0)$
 independent of $ {\mathcal F}_{T_0(n)}$.
A similar argument proves the existence and uniqueness of the
solution $\bar{u}_n$ to an equation
 similar to \eqref{u_n} with $u_0$ and $W$
replaced by $\bar{u}_0$ and $\bar{W}$ respectively. Hence,
\eqref{u_n} has a unique solution $u_n$
 on the interval $[0, 2 T_0(n)]$, defined
by $u_n(x,t) := \bar{u}_n(x,t-T_0(n))$ for $t\in [T_0(n), 2
T_0(n)]$. Since there exists $N\geq 1$ such that $N T_0(n) \geq T$
an easy induction argument concludes the proof.
\end{proof}
\subsection{Some bound for the stochastic integral}
We shall prove moment estimates for the  (space-time) uniform norm   for
$\mathcal{L}(u_n)$ which will be needed later.

We set
\[ 
\|\mathcal{L}(u_n)\|_{L^\infty}:=\sup_{t\in[0,T]}\sup_{x\in{\mathcal{D}}}|\mathcal{L}(u_n)(x,t)|.
\] 
\begin{lemma}
Let  $\sigma$ satisfy Condition \eqref{undif} with $\alpha \in (0,1]$, let Condition {\bf (C$_\alpha$)} hold,  
and let $u_n$ be the solution to 
the SPDE \eqref{u_n}. Furthermore, suppose that $q> \frac{2\alpha  d}{4-d}$. 
 Then for any  $p\in [1,\infty)$   there exists a positive constant $C_p(T)$ such that for every  $n\geq 1$, 
we have:
\begin{equation}\label{sup_L_n}
E\Big{(}\|\mathcal{L}(u_n)\|_{L^\infty}^{2p}\Big{)} \leq C_{p}(T) n^{2\alpha p} .
\end{equation}
\end{lemma}
\begin{proof}
Since $d<4$, $\alpha \in (0,1]$ and $q\in [3,\infty)$, we have $q>2\alpha$ and may choose $\tilde{q}\leq q$ with
$\tilde{q}>(2\alpha)\vee \frac{2\alpha d}{4-d}$. 
For $t\in [0,T]$, using the Burkholder-Davis-Gundy inequality,
\eqref{G-est-1}, the growth condition \eqref{undif} on $\sigma$  and H\"older's inequality
with conjugate exponents $\frac{\tilde{q}}{2\alpha}$ and
$\frac{\tilde{q}}{\tilde{q}-2\alpha}$, we obtain for any $t\in [0,T]$ and $x\in
{\mathcal{D}}$
\begin{align*}
E|{\mathcal L}  (u)(x,t)|^{2p}& 
 \leq  C_{p} E\Big| \int_{0}^{t} \int_{\mathcal{D}} (t-s)^{- \frac{d}{2}}
\exp(h(\cdot,t,s)) \chi_{n}(\| u_{n}(\cdot,s)\|_{q} ) [1+|u_n(y,s)|^{2\alpha}] dy ds \Big|^{p}\\
&  \leq   C_{p}  E\Big| \int_{0}^{t} \int_{\mathcal{D}} (t-s)^{-
\frac{d}{2}} \Big\|
\exp(h(\cdot,t,s))\Big\|_{\frac{\tilde{q}}{\tilde{q}-2\alpha}}
\chi_{n}(\| u_{n}(\cdot,s)\|_{q} )
[1+\|u_{n}(y,s)\|_{q}]^{2\alpha} ds \Big|^{p}\\
&\leq  C_{p}  \Big|  \int_{0}^{t} (t-s)^{-\frac{d}{2} + \frac{d (\tilde{q}-2\alpha)}{4\tilde{q}}} n^{2\alpha} ds \Big|^{p}
 \leq C_{p} n^{2\alpha p},
\end{align*}
where as above we let $h(x,t,s)$ be defined by \eqref{undif}. 
The last
inequality holds provided that
$-\frac{d}{4}
\Big{(}1+\frac{2\alpha}{\tilde q}\Big{)}>-1$  which holds true since 
 $ q\geq \tilde{q} >\frac{2\alpha d}{4-d}\vee (2\alpha)$. 

Similar computations using \eqref{G-est-1},  \eqref{G-est-2}
and the Taylor formula imply that for $x, \xi \in {\mathcal{D}}$
and $t\in [0,T]$, we have for $\lambda \in (0,1)$, $\tilde{q}\leq
q$, $p\in [1,\infty)$  and $n\geq 1$:
\begin{align}
E|{\mathcal L} & (u)(x,t)- {\mathcal L}  (u)(\xi,t)|^{2p} \leq  C_p
E\Big| \int_0^t \int_{\mathcal{D}} |G(x,y,t-s) - G(\xi,y,t-s)|^2  \chi_{n}(\| u_{n}(\cdot,s)\|_{q} ) \nonumber \\
&\qquad \times \big[1 + |u_n(y,s)|^{2\alpha} \big] \, dy \, ds \Big|^p \\
\leq & C_p |x-\xi|^{2\lambda p} E\Big| \int_0^t (t-s)^{-\frac{(d+1)\lambda }{2} } (t-s)^{-\frac{d}{2}(1-\lambda)}
\chi_{n}(\| u_{n}(\cdot,s)\|_{q} ) \nonumber \\
&\qquad \times \|
\exp(h(\cdot,t,s))\|_{\frac{\tilde{q}}{\tilde{q}-2\alpha}}
[1+\|u_{n}(\cdot,s)\|_{\tilde{q}}]^{2\alpha} ds \Big|^{p} \nonumber \\
\leq &  C_p |x-\xi|^{2\lambda p} n^{2\alpha p} \Big| \int_0^t (t-s)^{-\frac{d+\lambda}{2}
+ \frac{d (\tilde{q}-2\alpha)}{4\tilde{q}}} ds \Big|^p\nonumber \\
\leq & C_p(T) n^{2\alpha p} |x-\xi|^{2\lambda p}, \label{supx}
\end{align}
provided that  $ -\frac{d+\lambda}{2} + \frac{d(\tilde{q}-2\alpha)}{4\tilde{q}}>-1$, which holds true if 
$0\leq \lambda < \Big{(}2-\frac{d}{2}\Big{)}\wedge 1  $ and $ \tilde{q} > \frac{2\alpha d}{4-d-2\lambda}$. 
Hence, for $q>\frac{2\alpha d}{4-d}$ one can find $\lambda \in (0,1) $ small enough  and
$\tilde{q}$ as above.

Using again the Taylor formula,   \eqref{G-est-1} and
\eqref{G-est-3}, we obtain, for $0\leq t' \leq t \leq T$ and
 $\mu \in [0,1]$
\begin{align}
E|{\mathcal L} & (u)(x,t)- {\mathcal L}  (u)(x,t')|^{2p} \leq  C_p
E \Big| \int_0^t \int_{\mathcal{D}}  \big[  |G(x,y,t-s)|^{2(1-\mu)} + |G(x,y,t'-s)|^{2(1-\mu)} \big] \nonumber \\
&\qquad \times  |G(x,y,t-s)-G(x,y,t')|^{2\mu}
\chi_{n}(\| u_{n}(\cdot,s)\|_{q} ) [1+ |u_n(y,s)|^{2\alpha}] dy ds \Big|^p \nonumber \\
\leq & |t-t'|^{2\mu p}  E\Big| \int_0^t (t-s)^{-2\mu
(\frac{d}{4}+1) - (1-\mu) \frac{d}{2}} \| \exp(
h(\cdot,t,s))\|_{\frac{\tilde{q}}{\tilde{q}-2\alpha}}
\nonumber  \\
&\qquad \quad \times \big( 1+\|u_n(\cdot,s)\|_{\tilde{q}}^{2\alpha}\big)
\chi_n(\|u_n(\cdot,s)\|_q) ds \Big|^p \nonumber \\
\leq & C_p(T)   |t-t'|^{2\mu p} n^{2\alpha p},\label{supt}
\end{align}
where the last inequality holds true if 
$-\frac{d}{2} -2\mu +
\frac{d}{4} \Big{(} 1-\frac{2\alpha}{\tilde q}\Big{)} > -1$;  this is similar to the previous 
requirement used to prove \eqref{supx} 
replacing $\frac{\lambda}{2}$ by $2\mu$. 
Thus, since  $q>\frac{2\alpha d}{4-d}$, we may find $\tilde{q} \in
(\frac{2\alpha d}{4-d} , q]$ and $\mu \in(0,1)$ which satisfy
this constraint, and such that \eqref{supt} holds for any $p\in [1,+\infty)$.

The upper estimates \eqref{supx}, \eqref{supt} imply  the existence of some positive constants 
$\lambda$ and $\mu$, and given  $p\in [1,\infty)$ of some positive constant $C_p(T)$ (independent of
$n$) such that 
for $x,x'\in {\mathcal{D}}$ and $t,t'\in [0,T]$, we have for every $n\geq 1$  
 \[ E|{\mathcal L}  (u)(x,t)- {\mathcal L}  (u)(x',t')|^{2p} \leq  C_p(T)
\big[ |x-\xi|^{2\lambda p} +  |t-t'|^{2\mu p}\big]  n^{2\alpha p}  . \]

 Therefore, the Garsia-Rodemich-Rumsey Lemma
 yields  the upper estimate \eqref{sup_L_n}.
\end{proof}

\subsection{Galerkin approximation}
 In this section we need  some stronger integrability condition on  the
initial condition $u_0$, which is required to be in $L^4({\mathcal D})$. More precisely, we suppose that the
following condition ($\tilde{\mbox{\bf C}_\alpha}$) is satisfied.
\smallskip

\noindent {\bf Condition ($\tilde{\mbox{\bf C}_\alpha}$)} {\it One of the following properties is satisfied:

(i) Either $d=1,2$ and $q\in [4,\infty)$, \quad or $d=3$ and $q\in [6,\infty)$; 

(ii) $d=3$ and $q\geq 4$ is such that $q\in \big( 6(1-\alpha)\vee (6\alpha) , 6\big)$. 

\smallskip

\noindent Note that if Condition ($\tilde{\mbox{\bf C}_\alpha}$) is satisfied, we have $\|\cdot \|_4\leq C \| \cdot \|_q$ for some positive
constant $C$. 
}

For any $n\geq 1$, we define $$v_n:=u_n-{\mathcal L}(u_n).$$ Then,
formally, $v_n$ satisfies the following equation:
 \begin{align}\label{v_n_formal}
& \partial_t v_n+ [\Delta^2-\Delta]  v_n - (\Delta - Id)
\Big{(}\chi_n(\|v_n+\mathcal{L}(u_n)\|_q)f(v_n
+\mathcal{L}(u_n)\Big{)} =0\;{\rm
in}\;{\mathcal{D}}\times[0,T), \\
&v_n(x,0)=u_0(x)\;\;{\rm in}\;\;{\mathcal{D}}, \nonumber \\
&\frac{\partial v_n}{\partial \nu}=\frac{\partial\Delta
v_n}{\partial \nu}=0\;\;{\rm on}\;\;\partial{\mathcal{D}}\times
[0,T). \nonumber 
\end{align}
 For a strict definition of solution, we say that $v_n$ is a weak
solution of the above equation \eqref{v_n_formal} 
if for all $\phi\in {\mathcal
C}^4({\mathcal{D}})$ with $\frac{\partial\phi}{\partial
\nu}=\frac{\partial\Delta \phi}{\partial \nu}=0$ on $\partial
{\mathcal{D}}$, we have:
\begin{align*}   
\int_{\mathcal{D}}\Big(v_n(x,t)-u_0(x)\Big) &\phi(x)\, dx
=  \int_0^t \int_{{\mathcal{D}}}\Big\{\big[ -\Delta^2 + \Delta \big] \phi(x)\, v_n(x,s) \\
& +\big[ \Delta\phi(x) - \phi(x)\big] \chi_n(\|v_n+\mathcal{L}(u_n)\|_q)f(v_n+ {\mathcal L}(u_n)) \Big\}\; dx ds. \nonumber 
\end{align*}  
Using the Green's function $G$ defined by \eqref{Green}, we deduce
the integral form of this equation: 
\begin{align}\label{integral-form-v_n}
v_n(x,t) = \int_{{\mathcal{D}}} &  u_0(y) G(x,y,t)\;dy
+\int_0^t \int_{{\mathcal{D}}} \big[ \Delta G(x,y,t-s) - G(x,y,t-s)\big]\nonumber \\
&  \times \chi_n(\|v_n+\mathcal{L}(u_n)\|_q)
f\big(v_n(y,s)+ {\mathcal L}(u_n)(y,s)\big)   \;dyds . 
\end{align}

We will use the Galerkin method to prove the existence of the
solution $v_n$ for the equation \eqref{v_n_formal}. Let us denote
by
 $0=\lambda_0<\lambda_1\leq\lambda_2\leq\cdots$  the
eigenvalues of Neumann Laplacian operator inducing
$\{w_i\}_{i=0}^\infty$ as an orthonormal basis of
$L^2({\mathcal{D}})$ of eigenfunctions, i.e., $(w_i,
w_j)_{L^2({\mathcal{D}})}=\delta_{ij}$ and
\begin{equation}
-\lambda_i w_i=\Delta w_i\;\;{\rm in} \;\;{\mathcal{D}},\;\;\;\;
\frac{\partial w_i}{\partial \nu}=0\;\;{\rm on}\;\; \partial
{\mathcal{D}}\;\;{\rm for}\;\; i=0,1,2,\cdots.
\end{equation}
 Let $P_m$ denote the
orthogonal projection from $L^2(\mathcal{D})$ onto ${\rm span}\{
w_0, w_1, \cdots, w_m \}.$
For every $m=0,1,2,\cdots$ we consider the function $v_n^m$
$$ v_n^m(x,t)=\sum_{i=0}^m \rho_i^m(t)w_i(x),$$
defined by the Galerkin ansatz, where
\begin{equation}   \label{s-1}
\left\{ 
\begin{array}{l} \frac{\partial}{\partial t} v_n^m  +\big( \Delta^2 - \Delta \big)
v_n^m- \big( \Delta - Id \big) \Big[\chi_n(\|v_n^m
+\mathcal{L}(u_n)\|_q)\, P_m \big(  f(v_n^m+\mathcal{L}(u_n)) \big)\Big]  =0,  \\ 
v^n_m(x,0)=P_m(u_0) \;\;{\rm in}\;\;{\mathcal{D}},\;\;\;
 \frac{\partial v^m_n}{\partial \nu}=\frac{\partial\Delta
v^m_n}{\partial \nu}=0\;\;{\rm on}\;\;\partial{\mathcal{D}}.
\end{array}
\right.
\end{equation}

This yields an  initial value problem of ODE satisfied by
$\rho_i^m(t)$ for $i=0, 1,\cdots , m$. By  standard arguments of
ODE, this initial value problem has a local solution. We will show
that a global solution exists.

Multiplying by $v_n^m$  both sides of  $\eqref{s-1},$ we obtain
\begin{align}\label{est-1}
\frac{1}{2}\frac{d}{dt} & \|v_n^m(\cdot,t)\|_{2}^2 +\|\Delta v_n^m(.,t)\|_{2}^2 + \| \nabla v_n^m(\cdot,t)\|_2^2 \nonumber \\
& =\chi_n(\|v_n^m(\cdot,t)+\mathcal{L}(u_n)(\cdot,t)\|_q)
\int_{\mathcal{D}}
f\big(v_n^m(x,t)+\mathcal{L}(u_n)(x,t)\big)  \big[ \Delta v^m_n(x,t) - v^m_n(x,t)\big]  \;dx \nonumber \\
&= \sum_{i=1}^3 T_i(t),
\end{align}
where
\begin{eqnarray*}
T_1(t)&=&  \chi_n(\|v_n^m(\cdot, t)+\mathcal{L}(u_n)(\cdot,t)\|_q)
\int_{\mathcal{D}}
\big[ f\big(v^m_n(x,t) + {\mathcal L}(u_n)(x,t)\big) - f\big( v^m_n(x,t)\big) \big]\\
&&\qquad  \times
 \big[ \Delta v^m_n(x,t) - v^m_n(x,t)\big] dx , \\
T_2(t)&=&  \chi_n(\|v_n^m(\cdot,t)+\mathcal{L}(u_n)(\cdot,t)\|_q)
\int_{\mathcal{D}}
f\big( v^m_n(x,t)\big) \Delta v^m_n(x,t) dx , \\
T_3(t)&=& - \chi_n(\|v_n^m(\cdot,t)+\mathcal{L}(u_n)(\cdot,t)\|_q)
\int_{\mathcal{D}} f\big( v^m_n(x,t)\big)  v^m_n(x,t) dx .
\end{eqnarray*}
Since $f$ is a polynomial of degree 3, then we have for $x,y\in {\mathbb R}$: 
$$|f(x+y)-f(x)|\leq
c|y|(1+x^2+y^2).$$   
Thus by Cauchy-Schwarz and Young inequality
we obtain for any $\varepsilon>0$
\begin{align*}
T_1(t)\leq & \; C
\chi_n(\|v_n^m(\cdot)+\mathcal{L}(u_n)(\cdot,t)\|_q)
\int_{\mathcal{D}}
|{\mathcal L}(u_n)(x,t)| \,  \big[ 1+|v^m_n(x,t)|^2 + |{\mathcal L}(u_n)(x,t)|^2\big]\\
&\qquad \times  \big[ |\Delta v^m_n(x,t)| + |v^m_n(x,t)|\big] dx\\
\leq & \; C \;
\chi_n(\|v_n^m(\cdot,t)+\mathcal{L}(u_n)(\cdot,t)\|_q) \|{\mathcal
L}(u_n)(\cdot,t)\|_\infty
\big[ 1+\|v^m_n(\cdot,t)^2\|_2+ \|{\mathcal L}(u_n)(\cdot,t)^2\|_2\big] \\
& \qquad \times \big[  \|\Delta v^m_n(\cdot ,t)\|_2 + \|v^m_n(\cdot ,t)\|_2  \big]\\
\leq &\;  \varepsilon   \big[  \|\Delta v^m_n(\cdot ,t)\|_2^2 + \|v^m_n(\cdot ,t)\|_2^2  \big]\\
&+ \frac{C}{\varepsilon}
\chi_n(\|v_n^m(\cdot,t)+\mathcal{L}(u_n)(\cdot,t)\|_q)
 \|{\mathcal L}(u_n)(\cdot,t)\|_\infty^2 \big[ 1 + \|v^m_n(\cdot,t)\|_4^4 + \| {\mathcal L}(u_n)(\cdot,t)\|_4^4\big].
\end{align*}
Observe that $f(x)=ax^3 + g(x)$, where $a>0$ and $g$ is a
polynomial of degree $2$.
 Hence, $f'(x)=3 a x^2 + 2bx +c$ for some real constants
$b$, $c$, and $f'(x)\geq 2ax^2 - \tilde{c}$ for some non-negative
constant $\tilde{c}$. So, an integration by parts yields for any
$\varepsilon >0$
\begin{align*}
T_2&(t) =-\chi_n(\|v_n^m(\cdot,t)+\mathcal{L}(u_n)(\cdot,t)\|_q)
\int_{\mathcal{D}}
f'(v^m_n(x,t)) |\nabla v^m_n(x,t)|^2 dx\\
\leq & \; C \chi_n(\|v_n^m(\cdot,t)+\mathcal{L}(u_n)(\cdot,t)\|_q)
\int_{\mathcal{D}}
\big[ -2 a |v^m_n(x,t)|^2 |\nabla v^m_n(x,t)|^2 + \tilde{c} |\nabla v^m_n(x,t)|^2 \big] dx \\
\leq & \; - C
\chi_n(\|v_n^m(\cdot,t)+\mathcal{L}(u_n)(\cdot,t)\|_q)
\int_{\mathcal{D}}
v^m_n(x,t) \Delta v^m_n(x,t) dx\\
\leq &\;  \varepsilon \|\Delta v^m_n(\cdot,t)\|^2_2 +
\frac{C}{\varepsilon}
\chi_n(\|v_n^m(\cdot,t)+\mathcal{L}(u_n)(\cdot,t)\|_q)
\|v^m_n(\cdot,t)\|_2^2.
\end{align*}
Finally, since $x f(x) \geq \frac{7}{8} a x^4 - \tilde{C}$ with
$a$, $\tilde{C}>0$, we obtain
\[ T_3(t)\leq  \chi_n(\|v_n^m(\cdot,t)+\mathcal{L}(u_n)(\cdot,t)\|_q)
\Big[ \int_{\mathcal{D}} - \frac{7}{8} a |v^m_n(x,t)|^4 dx
 + \tilde{C}|{\mathcal{D}}| \Big] . \]
The above upper estimates of $T_i(t)$, $i=1,2,3$,  imply that for $\varepsilon >0$ small enough,
\begin{align} \label{d-dt}
\frac{1}{2}& \frac{d}{dt}\|v_n^m(\cdot,t)\|_{2}^2  + \frac{1}{2}
\|\Delta v_n^m(\cdot,t)\|_{2}^2 +  \| \nabla v^m_n(\cdot,t)\|_2^2
\leq
C \chi_n(\|v_n^m(\cdot,t)+\mathcal{L}(u_n)(.,t)\|_q  )\nonumber \\
&\times  \Big( \|\mathcal{L}(u_n(\cdot,t))\|_\infty^2 \big[
1+\|v_n^m(\cdot,t)\|_{4}^4+\|\mathcal{L}(u_n(\cdot,t))\|_4^4\big]
+ \|v^m_n(\cdot,t)\|_2^2 +1\Big).
\end{align}

Since 
$$\|v_n^m(\cdot,0)\|_{2}=\|P_mu_0\|_{2}\leq\|u_0\|_{2},$$
 integrating \eqref{d-dt} from $0$ to  $t\in(0,T] $, 
and using H\"older's and Young's inequality, we obtain
\begin{align}\label{est-v_n^m}
 \| v_n^m(\cdot, t)\|_{2}^2 &
+\int_0^t\|\Delta v_n^m(\cdot, s)\|_{2}^2\;ds \leq  \|u_0\|_{2}^2
+CT\Big(1+\|\mathcal{L}(u_n)\|^6_{L^\infty}\Big)
\nonumber \\
&+C\Big(1+\|\mathcal{L}(u_n)\|_{L^\infty}^2\Big)\int_{0}^t
\chi_n(\|v_n^m(\cdot,s)+\mathcal{L}(u_n)(\cdot,s)\|_q)  \big( \|v_n^m(\cdot,s)\|_{4}^4+1\big) \;ds.
\end{align}
Since we have imposed $q\in [4,\infty)$,  if the cut-off function $\chi_n$ applied to the $\|.\|_q$
norm of some function $U$ is not zero, we deduce that $\|U\|_4 \leq C(n+1) \leq Cn$. Recall that $|\chi_n|\leq 1$; 
thus the triangular inequality yields
\begin{align}\label{est^3rd^term}
 \int_{0}^t\!\!   \chi_n(\|v_n^m(\cdot,s) & +\mathcal{L}(u_n)(\cdot,s)\|_q)  \big( \|v_n^m(\cdot,s)\|_{4}^4 +1\big) ds
\leq  C T \big( 1+ \|\mathcal{L}(u_n)\|^4_{L^\infty} \big)  
\nonumber \\
 & \qquad \qquad
 +  C \int_{0}^t\!\! \chi_n(\|v_n^m(\cdot,s)+\mathcal{L}(u_n)(\cdot,s)\|_q)
 \|v_n^m(\cdot,s)+\mathcal{L}(u_n)(\cdot,s)\|_{4}^4\;ds \nonumber \\
\leq & C T  \big( 1+\|\mathcal{L}(u_n)\|_{L^\infty}^4+ n^4 \big).
\end{align}

The upper estimates \eqref{est-v_n^m} and \eqref{est^3rd^term}
imply that
\begin{align}
&\sup_{t\in[0,T]}\|v_n^m(\cdot,t)\|_{2}^2 \leq \|u_0\|_2^2 + C
\big( 1+\|{\mathcal L}(u_n)\|_{L^\infty}^6\big)
+ C n^4 T \big( 1+  \|{\mathcal L}(u_n)\|_{L^\infty}^4   \big)  ,\nonumber \\
&\int_0^T [\|v^m_n(\cdot,t)\|_2^2 + \|\Delta v^m_n(\cdot,t)\|_2^2]
dt \leq  (T+1) \|u_0\|^2 + C (T^2+1) (1+\|{\mathcal
L}(u_n)\|_{L^\infty}^6)
\nonumber \\
&\qquad \qquad \qquad \qquad\qquad \qquad \qquad \qquad
+ C n^4 (T^2+1) (1+\|{\mathcal L}(u_n)\|_{L^\infty}^2). \label{estimH2O}
\end{align}
Since the $H^{2}({\mathcal{D}})$-norm  is equivalent to
$\Big{(}\int_{{\mathcal{D}}} \big( |\Delta u(x)|^2+|u(x)|^2\big) \;
dx\Big{)}^{\frac{1}{2}}$ under the boundary condition
$\frac{\partial u}{\partial \nu}=\frac{\partial\Delta u}{\partial
\nu}=0$ on $\partial{\mathcal{D}}$, the right-hand side of
\eqref{estimH2O} depends on $n$ but  is independent of the index
$m$. Thus, a standard weak compactness argument proves that for
fixed $n$,
 as $m\to \infty$, a subsequence of  $(v^m_n, m\geq 1)$ converges weakly
in $ L^2(0,T;H^2({\mathcal{D}}))$ to a solution $v_n$ of
\eqref{v_n_formal} with homogeneous Neumann  boundary conditions.

Let  $(\epsilon_k)$ denote the orthonormal basis defined in
Section \ref{secGreen} and set
\begin{equation}\label{B0}
B(u_0):=\frac{1}{2}\Big\|\sum_{k\in {\mathbb N}^{d}}
 [\lambda_k+1]^{-\frac{1}{2}}(u_0,\epsilon_k)_{L^2({\mathcal{D}})}\epsilon_k \Big\|_{2}^2
= \frac{1}{2}   \sum_{k\in {\mathbb N}^{d}}[\lambda_k+1]^{-1}
(u_0,\epsilon_k)_{L^2({\mathcal{D}})}^2 .
\end{equation}

Note that if ${\mathcal{D}}$ is the unitary cube,  we have
$B(u_0)\leq \frac{1}{2} \|u_0\|_2^2$. The following lemma
provides estimates of the $L^4({\mathcal{D}})$-norm of $u_n$.

\begin{lemma}\label{bound_u_n}
Let $\sigma$ be Lipschitz and satisfy the sub linearity condition \eqref{undif} with $\alpha \in (0,1]$,  and let 
$u_0\in L^q({\mathcal D})$ where $q$ satisfies Condition    {\bf ($\tilde{ \mbox{\bf C}_\alpha}$)}.  Let $u_n$ be the solution to 
the SPDE \eqref{u_n} and $B(u_0)$ be defined by \eqref{B0}. Then  
there exists a constant $C:=C(t,{\mathcal{D}})$ independent of the
index $n$ satisfying
\begin{equation}\label{bound-u_n-siki}
\int_0^t \!\!\chi_n(\|u_n(\cdot,s)\|_q)\|u_n(\cdot,s)\|_{4}^4ds
\leq C \Big\{\! 1+ B(u_0) + \!\int_0^t \!\!
\chi_n(\|u_n(\cdot,s)\|_q) \|\mathcal{L}(u_n)(\cdot,s)\|_4^4 ds
\Big\} .
\end{equation}
\end{lemma}
\begin{proof}
 Using the orthonormal basis $(\epsilon_k)$ defined at the beginning of Section \ref{secGreen}, we write
$v_n\in L^2({\mathcal{D}})$ as
$$v_n(x,t)=\displaystyle{\sum_{k\in {\mathbb N}^d}}
\rho_k(t)\epsilon_k(x).$$

To ease  notation, for $x\in {\mathcal D}$ and $s\in [0,T]$  we set
$$Q(x,s):=\chi_n(\|v_n(\cdot,s)+{\mathcal L}(u_n)(\cdot,s)\|_q)
f(u_n(x,s)).$$ Then the equation \eqref{v_n_formal} is written as
follows
\begin{equation}\label{eq2}
\partial_t v_n+ (-\Delta+Id)(-\Delta)v_n +(-\Delta+Id)Q=0,
\end{equation}
with the boundary conditions $v_n(x,0)=u_0(x)$  in
${\mathcal{D}}$, and $\frac{\partial v_n}{\partial
\nu}=\frac{\partial\Delta v_n}{\partial \nu}=0$ on
$\partial{\mathcal{D}}\times [0,T)$.

We set $A=-\Delta + Id$, apply $A^{-1}$ to \eqref{eq2} and take
the $L^2$ inner product with $v_n(\cdot,t)$. The
$L^2$-orthogonality of the eigenfunctions $\epsilon_k$ of
$\Delta$ gives
\[ \sum_{k\in {\mathbb N}^d} [\lambda_k+1]^{-1}\rho_k(t) \partial_t \rho_k(t)
+ \sum_{k\in {\mathbb N}^d} \lambda_k \rho_k(t)^2 + \big(
Q(\cdot,t), v_n(\cdot,t)\big) =0.\] Integrating this identity from
$0$ to $t$ yields
\[ \int_0^t \big( Q(\cdot,s), v_n(\cdot,s)\big) ds =
 \sum_{k\in {\mathbb N}^d}\frac{1}{2}  [\lambda_k+1]^{-1} \big( \rho_k(0)^2 - \sum_{k\in {\mathbb N}^d}
\Big[  [\lambda_k+1]^{-1} \big( \rho_k(t)^2 + \int_0^t \lambda_k \rho_k(s)^2 ds \Big].
\]
Since $\lambda_k\geq 0$ for all $k$, we obtain
\begin{equation}\label{siki-3}
\int_0^t\big(Q(\cdot,s), v_n(\cdot,s))_{L^2}\;ds \leq
\frac{1}{2}\sum_{k\in {\mathbb
N}^d}[\lambda_k+1]^{-1}\rho_k^2(0)=B(u_0).
\end{equation} %
%
Furthermore, $f$ is a polynomial of degree 3; therefore,  $f(u_n)\geq
\frac{4}{5} u_n^4 -c$ for some non negative constant $c$. This yields  
\[
\int_0^t \!\! \chi_n(\| u_n(\cdot,s) \|_q)\|u_n(\cdot,s)\|_4^4 ds
\leq C \Big\{ 1+\int_0^t\!\! \int_{{\mathcal{D}}}\!\!
\chi_n(\|u_n(\cdot,s)\|_q)f(u_n(\cdot,s))u_n(\cdot,s) dxds\Big\}.
\]
Since $Q=\chi_n(\|v_n+\mathcal{L}(u_n)\|_q) f(u_n)$ and $u_n=\mathcal{L}(u_n)+v_n$, using \eqref{siki-3} in the
previous identity we obtain 
\begin{align}\label{siki-4}
\int_0^t \!\! \chi_n(\| u_n(\cdot,s) \|_q)\|u_n(\cdot,s)\|_4^4\,
ds
\leq  C\Big\{  & 1+\int_0^t\!\! \int_{{\mathcal{D}}} \chi_n(\|
u_n(\cdot,s) \|_q)
f(u_n(\cdot,s))\mathcal{L}(u_n(\cdot,s))\,dxds \nonumber \\
& +B(u_0)\Big\}.
\end{align}
Using once more the fact that  $f(u_n)$ is a third degree polynomial, Young's inequality implies that
for any $\epsilon >0$ and $s\in [0,T]$,
\begin{align*}  \int_{{\mathcal{D}}} &  \chi_n(\| u_n(\cdot,s) \|_q)
f(u_n(x,s))  \mathcal{L}(u_n(x,s))\,dx
 \leq \epsilon  \int_{{\mathcal{D}}}  \chi_n(\| u_n(\cdot,s) \|_q)
|f(u_n(.,s))|^{4/3} dx \\
& \qquad+ \frac{C}{\epsilon} \int_{{\mathcal{D}}}  \chi_n(\|
u_n(\cdot,s) \|_q)
| \mathcal{L}(u_n(x,s))|^4 \,dx\\
&\leq C \epsilon  \chi_n(\| u_n(\cdot,s) \|_q)
\|u_n(\cdot,s)\|_4^4 + C + \frac{C}{\epsilon} \chi_n(\|
u_n(\cdot,s) \|_q) \| \mathcal{L}(u_n(\cdot,s))\|_4^4.
\end{align*}
Consequently, plugging this upper estimate in \eqref{siki-4} and
choosing $\epsilon$ small enough, we complete the proof of 
 \eqref{bound-u_n-siki}.
\end{proof}
\subsection{Existence of a global solution}
The Sobolev embedding theorem implies that for $d=1,2,3$,
$H^2({\mathcal{D}}) \subset L^4({\mathcal{D}})$. Hence,
computations similar to that used to prove  \eqref{est-v_n^m}
with the weak $H^2({\mathcal{D}})$-limit $v_n$ of $v_n^m$ taken
instead of $ v_n^m$, show  that for any $\epsilon >0$ we have 
\begin{align*}  
 \|  v_n(\cdot, t)\|_{2}^2 &
+\int_0^t\|\Delta v_n(\cdot, s)\|_{2}^2\;ds
\leq \|u_0\|_{2}^2 + C \int_0^t \tilde{T}_1(s) ds + \epsilon \int_0^t\|\Delta v_n(\cdot,s)\|^2_2 ds \\
& + C (1+{\epsilon}^{-1}) \Big[ T+  \int_0^t
\chi_n(\|v_n(\cdot,s)+{\mathcal L}(u_n)(\cdot,s))\|_q)
\|v_n(\cdot,s)\|_4^4 ds \Big],
\end{align*}
where
the  Cauchy-Schwarz and Young inequalities yield for any $\epsilon >0$ 
\begin{align*}
\tilde{T}_1(s) \leq & \epsilon \big[ \|\Delta v_n(\cdot,s)\|^2_2 +
\| v_n(\cdot,s)\|_2^2\big] + \frac{C}{\epsilon}
\chi_n(\|v_n(\cdot,s)+\mathcal{L}(u_n)(\cdot,s))\|_q)
\bar{T}_1(s),
\end{align*}
for $\bar{T}_1(s)$  defined by
\begin{align*}
 \bar{T}_1(s) : = &\int_{\mathcal{D}} |{\mathcal L}(u_n(x,s))|^2
 \big[ 1+ |v_n(x,s)|^4 +  |{\mathcal L}(u_n(x,s))|^4\big] dx\\
\leq & C\big[ 1+ \|{\mathcal L}(u_n(\cdot,s))\|_{\infty}^6 +
\|{\mathcal L}(u_n(\cdot,s))\|_{\infty}^2\|v_n(\cdot,s)\|_4^4
\big].
\end{align*}
Recall that $u_n=v_n+{\mathcal L}(u_n)$. Choosing $\epsilon$ small enough, using the Gronwall Lemma and
Lemma \ref{bound_u_n}, we deduce that for $t\in [0,T]$ there exists a positive constant $C(T)$ such that 
\begin{align} \label{est-v_n-last}
 \|  v_n& (\cdot, t)\|_{2}^2 + \frac{1}{2} \int_0^t \|\Delta v_n(\cdot,s)\|_2^2 ds
\leq C(T) \Big( \|u_0\|_2^2  + 1 \nonumber \\
&\qquad  + \int_0^t\!\! \chi_n(\|u_n(\cdot,s)\|_q)
\big[ 1+ \|{\mathcal L}(u_n(\cdot,s))\|_6^6 +  \|{\mathcal L}(u_n)\|_{L^\infty}^2 \; \|v_n(\cdot,s)\|_4^4  \big] ds \Big)  \nonumber  \\
&\leq  C(T) \Big[ \|u_0\|_2^2 +1   + \|{\mathcal L}(u_n)\|_{L^\infty}^6\nonumber \\
&\qquad  +\big( 1+ \|{\mathcal L}(u_n)\|_{L^\infty}^2\big)
 \int_0^t \!\!  \chi_n(\|u_n(\cdot,s)\|_q)
 \|u_n(\cdot,s)\|_4^4   
\; ds   \Big] \nonumber  \\
&\leq C(T) \Big[ 1+  \|u_0\|_2^2 +    \|{\mathcal
L}(u_n)\|_{L^\infty}^6 + (1+\|  {\mathcal
L}(u_n)\|_{L^\infty}^2\big) B(u_0) \Big] .
\end{align}
holds for every $n\geq 1$. 

The following result is similar to (2.33)  in \cite{CW}, but the
proof is different. Note that a gap in this reference has been
fixed and that the proof is simpler using the Sobolev embedding
Theorem and \eqref{sup_L_n}; it does not rely any more on an
interpolation argument. As pointed out in \cite{AK}, where
stochastic existence for the Cahn-Hilliard equation has been
proven for bounded domains of general geometry, the important
property of the domain's boundary is being Lipschitz in dimensions
$1,2,3$. This fact together with the  above $H^2$-norm 
estimate \eqref{est-v_n-last}  allows us to use  easier
$L^\infty$-norm arguments.
\begin{lemma}
 Let  $\sigma$ be Lipschitz and  satisfy the sub linear growth condition  \eqref{undif} with $\alpha \in (0,1]$. 
Let  $u_0\in L^q({\mathcal D})$ where $q$ satisfies Condition   {\bf ($\tilde{\mbox{\bf C}_\alpha}$)},  
 and  let $u_n$ be the solution to the SPDE \eqref{u_n}. Then for any
 $\beta \in (1,\infty)$ and $a>0$ such that $a\beta \in [2,\infty)$,
 there exists a positive constant $C(T)$ such that for every $n\geq 1$ we have
\begin{align}\label{sup_u_n}
E\Big{(}\Big{[}\int_0^T\|u_n(\cdot, t)\|_q^a\; dt\Big{]}^\beta\Big{)}& \leq C(T) \big[ 1+ \|u_0\|_2^{a\beta} + n^{3a \alpha \beta}
+ B(u_0)^{a\beta/2} (1+ n^{a\alpha\beta} )\big]  \nonumber \\
&\leq C(T) \big[1 +  \|u_0\|_2^{a\beta} + B(u_0)^{3a\beta/2} +
n^{3a \alpha \beta} \big].
\end{align}
\end{lemma}
\begin{proof}
For any integer $k\geq 1$, set
$\|u\|_{H^k({\mathcal{D}})}:=\big{(}\sum_{|a|\leq
k}\|D_x^au\|^2_{L^2({\mathcal{D}})}\Big{)}^{1/2}$. Using the
Sobolev inequality (see e.g. \cite{Brenner}, Theorem 1.4.6), we
deduce the existence of a positive constant  $C$ such that for
every $u\in H^k({\mathcal{D}})$, $ \|u\|_\infty \leq C
\|u\|_{H^k({\mathcal{D}})}$, provided that ${\mathcal{D}}$ is a
bounded domain with Lipschitz boundary and $k$ is an integer with
$k>d/2$. Therefore, if $\mathcal{D}$ is a unit cube of ${\mathbb
R}^d$, $d=1,2,3$, we have
\[ \|u\|_\infty \leq C \|u\|_{H^2({\mathcal{D}})}.\]
Hence, given any $a\in (0,+\infty)$ and $\beta \in [1,+\infty)$ with
$a\beta \geq 2$, since $u_n=v_n + {\mathcal L}(u_n)$,  and
${\mathcal{D}}$ is bounded and of Lipschitz boundary, then there
exists a positive constant $C$ depending on $T$,
$|{\mathcal{D}}|$, $a$ and $\beta$, such that for every integer
$n\geq 1$:
\begin{align*}
\Big( \int_0^T \|u_n(\cdot,t)\|_q^a dt \Big)^\beta & \leq C \Big[
\sup_{t\in [0,T]} \|v_n(\cdot,t)\|_q^{a\beta }
+ \sup_{t\in [0,T]} \|{\mathcal L}(u_n)(\cdot,t)\|_q^{a\beta } \Big]\\
&\leq C \Big[ \sup_{t\in [0,T]} \|v_n(\cdot,t)\|_\infty ^{a\beta }
+ \|{\mathcal L}(u_n)\|_{L^\infty}^{a\beta} \Big].
\end{align*}
Thus, the inequalities  \eqref{sup_L_n} and  \eqref{est-v_n-last} yield the existence of $C$ as above such that
\begin{align*}
E\Big[ \Big( \int_0^T \|u_n(.,t)\|_q^a dt \Big)^\beta  \Big]  \leq & C\big[ 1 + \|u(0)\|_2^{a\beta} + B(u_0)^{a\beta/2}
\big(1+ E\|{\mathcal L} (u_n)\|_{L^\infty}^{a\beta}\big)  + E\|{\mathcal L}(u_n)\|_{L^\infty}^{3a\beta}\big] \\
& + C(T) E\|  {\mathcal L}(u_n)\|_{L^\infty}^{a\beta}\\
\leq  & C\big[ 1 + \|u(0)\|_2^{a\beta} + B(u_0)^{a\beta/2} +  B(u_0)^{a\beta/2} n^{\alpha a\beta} + n^{3\alpha a\beta}\big] .
\end{align*}
This proves the first upper estimate in \eqref{sup_u_n}. The second one is a straightforward consequence of the Young inequality
applied with the conjugate exponents $3$ and $3/2$; this completes the proof.
\end{proof}

The above lemma provides an upper estimate of moments of the $q$-norm of ${\mathcal M}(u_n)$.
\begin{lemma}
Let $\sigma$ be Lipschitz
and satisfy the sub linearity condition \eqref{undif} with $\alpha \in (0,1]$, $u_0\in L^q({\mathcal D})$ where $q$ 
satisfies Condition   ($\tilde{\mbox{\bf C}_\alpha}$).  Let
 $u_n$ be the solution of the SPDE \eqref{u_n} and let $\beta \in [\frac{2q}{q-d}, \infty)$.   Then for ${\mathcal M}(u_n)$
 defined by \eqref{M}
 there exists a positive  constant  $C:=C(T)$  such that for and every $n\geq 1$
 the following estimate holds:
\begin{equation}\label{M_1-and-M_2}
E\Big(\sup_{0\leq t \leq T}\|
\mathcal{M}(u_n)(\cdot, t) \|_q^\beta\Big)  
\leq C\; \big[1+ \|u_0\|_2^{3\beta} + B(u_0)^{9\beta/4} + n^{9\alpha \beta} \big] .
\end{equation}
\end{lemma}
\begin{proof}Computations similar to that proving    \eqref{m_1-1-5}   yield for
$\frac{1}{\rho} + \frac{1}{r} = \frac{1}{q}+1$
\[
\|\mathcal{M}(u_n)(\cdot,t)\|_q\\
\leq C\; \int_0^t (t-s)^{-\frac{d+2}{4}+\frac{d}{4r}}
\Big{\|}\chi_n(\|u_n(\cdot,s)\|_q) f(u_n(y,s))\Big{\|}_{\rho} \; ds.
\]
Since $f$ is a third degree polynomial,  choosing $\rho=\frac{q}{3}$ as before, we deduce
\[
\Big\|\chi_n(\|u_n(\cdot,s)\|_q)f(u_n(y,s))\Big\|_{\frac{q}{3}}
\leq C \Big(1+\|u_n(\cdot , s)\|_q^3\Big).
\]
Thus for $\rho = \frac{q}{3}$ and $r$ such that $\frac{2}{q}+\frac{1}{r}=1$, we obtain for any $t\in [0,T]$:
\[
\|\mathcal{M}(u_n)(\cdot,t)\|_q \leq C \int_0^t
(t-s)^{-\frac{d+2}{4}+\frac{d}{4r}} \Big{(}1+\| u_n(\cdot,s)
\|_{q}^3\Big{)} \; ds.
\]
Let $\gamma\in (1,\infty)$ be such that $(-\frac{d+2}{4}+\frac{d}{4r})\gamma>-1$,
and $\gamma'$ be the conjugate exponent. Then $\gamma' > \frac{2q}{q-d}$ and H\"older's inequality yields
the existence of a positive constant $C$ such that
\[
\|\mathcal{M}(u_n)(\cdot,t)\|_q \leq
C(T)\Big\{1+\Big(\int_0^t \|u_n(\cdot, s)\|_q^{3\gamma'}\;
ds\Big)^{\frac{1}{\gamma'}}\Big\}.
\]
The upper estimate  \eqref{sup_u_n} completes the proof. Indeed, for $\beta > \frac{2q}{q-d}$ we can choose
$\gamma$ and $\gamma'$ as above with $\frac{\beta}{\gamma'} >1$; this clearly yields   $3\beta \geq 2$. Then 
H\"older's inequality  yields \eqref{M_1-and-M_2} for $\gamma' \beta \in [1,\infty)$.
\end{proof}
\quad\\
 In the above arguments; we  only assumed that the exponent $\alpha$ appearing in the
growth condition \eqref{undif} was in  the interval $(0,1]$, including the case of a usual linear growth condition. However, to prove 
that  equation \eqref{intf} has a unique global solution, we
must suppose that $\sigma$ has a sub linear growth/ More precisely, we have to assume that
$\alpha\in(0,1/9)$.

For every integer $n\geq 1$ let us define the stopping time $T_n$
as follows
\begin{equation}\label{defTn}
T_n:= \inf \Big\{ t\geq 0: \; \|u_n(\cdot,t)\|_q\geq n \Big\}.
\end{equation}
Then for every integer $n\geq 1$, the process $u(\cdot,t)=u_n(\cdot,t)$   is a solution of \eqref{intf} on the
interval $[0,T_n\wedge T]$. Assuming that $\alpha \in (0, \frac{1}{9})$, we will show that
$\displaystyle{\lim_{n\rightarrow\infty}}T_n=\infty$ a.s., which will enable us to solve
 \eqref{intf} on $[0,T]$ a.s. for any fixed $T$.
\begin{theorem} \label{thexistence}
Suppose that $\sigma$ is globally Lipschitz and satisfies the
sub-linearity condition \eqref{undif}  with $\alpha \in
(0,\frac{1}{9})$. Let $u_0\in L^q({\mathcal{D}})$ where $q$ satisfies Condition     ($\tilde{\mbox{\bf C}_\alpha}$). 
Then for any $T>0$ there exists a unique solution $u$ to the SPDE \eqref{intf}  in the time interval
$[0,T]$ (or equivalently
if $T_n$ is defined by \eqref{defTn}, $T_n \to  \infty $ a.s.  as $n\to \infty$); this solution   belongs to
$L^\infty \big( [0,T]; L^q({\mathcal{D}})  \big)$ a.s.
Furthermore, given any  $\beta \in \big( \frac{2q}{q-d},\infty\big)$,  we have
$$E\Big(1_{\{T\leq T_n\}} \sup_{t\leq T} \| u(.,t)\|_q^\beta\Big) \leq C(1+ n^{9\alpha \beta}).$$
\end{theorem}
\begin{proof}
The sequence $T_n$ is clearly non decreasing.
Fix $T>0$; by the definition of $T_n$,  on  the set  $\{T_n<  T\}$ we have for any $\beta \in [1,\infty)$
\[
\sup_{t\in [0,T]}\|u_n(\cdot,t)\|_{q}^{2\beta}\geq n^{2\beta}.
\]
Thus,  the  Chebyshev inequality,
\eqref{evolution-initial-value}, \eqref{u_nn}, \eqref{sup_L_n}
and \eqref{M_1-and-M_2} yield the existence of a constant $C$
depending on $T$, $\|u_0\|_q$ and $B(u_0)$ such that for every
$n\geq 1$ the next inequality holds true
\begin{equation}
P(T_n <  T)
\leq n^{-2\beta}E\Big(\sup_{t\in [0,T]}\|u_n(\cdot,
t)\|_q^{2\beta}\Big) \leq C  n^{-2\beta(1-9\alpha)}.
\end{equation}
 Since $\beta$ can be chosen large enough to ensure that $2\beta(1-9\alpha) >1$,  the Borel-Cantelli
Lemma implies that $P(\limsup_n \{ T_n<T\})=0$, that is $\lim_n
T_n \geq T$ a.s.  Since $T$ is arbitrary, this yields   $T_n \to
\infty$ a.s. as $n\to \infty$. The  uniqueness of the solution to
\eqref{u_n} implies that a process $u$ can be uniquely  defined
setting $u(\cdot,t)= u_n(\cdot,t)$ on $[0,T_n]$. Since $T_n\to
\infty $ a.s., we conclude that for any fixed $T>0$, equation
\eqref{intf} has a unique solution and the upper estimate of
moments of the $q$-norm of the solution
 follows from \eqref{evolution-initial-value}, \eqref{u_nn}, \eqref{sup_L_n} and \eqref{M_1-and-M_2}.
\end{proof}

\section{Generalization}\label{gener} The stochastic existence proof for the
Cahn-Hilliard/Allen-Cahn equation with noise could easily be
modified to hold for domains with more general geometry, cf. in
\cite{AK} the proposed eigenvalue-formulae-free approach for the
stochastic Cahn-Hilliard equation.

Our global existence and uniqueness result proven in Theorem \ref{thexistence} is also valid for the more
general model
\begin{equation}\label{S-CH-ACn}
\left\{ \begin{array}{rll}
u_t & =&-\varrho \Delta\Big(\Delta u-f(u)\Big)+\tilde{q}\Big(\Delta
u-f(u)\Big)
+\sigma(u)\dot{W} \quad {\rm in}\quad {\mathcal{D}}\times [0,T),\\
u(x,0)&=&u_0(x)\quad {\rm in}\quad{\mathcal{D}},\\
\frac{\partial u}{\partial \nu}&=&\frac{\partial \Delta u}{\partial
\nu}\; =0\;  \quad {\rm on}\quad \partial{\mathcal{D}}\times [0,T),
\end{array}
\right.
\end{equation}
for some constants $\varrho>0$ and $\tilde{q}\geq 0$.
The proof is very similar with the  following  simple modifications:
\begin{enumerate}
\item
We have to replace the Green's  function $G$ defined by
\eqref{Green} by the following $\varrho,\tilde{q}$-dependent one
$$ G^{\varrho,\tilde{q}}(x,y,t):=\sum_{k\in {\mathbb N}^d}
e^{(-\varrho\lambda_k^2+\tilde{q}\lambda_k)
t}\epsilon_k(x)\epsilon_k(y).$$
All the estimates used on $G$ also
hold for $G^{\varrho,\tilde{q}}$, since the operator
$-\varrho\Delta^2+\tilde{q}\Delta$  is parabolic in the sense of
Petrovsks\u {\i\i}.
\item
The estimate \eqref{bound-u_n-siki} also holds for $\varrho>0$ and
$\tilde{q}>0$ when $B(u_0)$ is defined as
\[
B(u_0):=\frac{1}{2}\Big\| \sum_{k\in {\mathbb N}^d}
[\varrho\lambda_k+\tilde{q}]^{-\frac{1}{2}}
(u_0,\epsilon_k)_{L^2({\mathcal{D}})}\epsilon_k\Big\|_{2}^2.
\]
Since $\varrho >0$ and $\lambda_k\geq 0$, one may also invert
$\varrho\lambda_k+\tilde{q}$ if $\tilde{q}>0$ for any $k\in{\mathbb N}^d$.

If $\tilde{q}=0$ (for $\varrho=1$ we get the Cahn-Hilliard equation) then
\[
B(u_0):=\frac{1}{2}\Big\| \sum_{k\in {\mathbb N^*}^d}
[\varrho\lambda_k]^{-\frac{1}{2}}
(u_0,\epsilon_k)_{L^2({\mathcal{D}})}\epsilon_k\Big\|_{2}^2.
\]
and $\varrho\lambda_k$, for any $k\in {\mathbb N^*}^d$, is
invertible.
\end{enumerate}

While the stochastic Cahn-Hilliard equation is a special case for
our analysis (with $\rho =1$ and $\tilde{q}=0$), this is not true for the stochastic Allen-Cahn
equation. In our model the assumption that $\varrho>0$ is crucial;
indeed, since  the fourth order operator is still acting, the
operator $-\rho \Delta^2 + \tilde{q}\Delta$ is also parabolic in
the sense of Petrovsk\u{\i\i}. Thus, the higher order
differential operator is dominating and all the upper estimates
of the Green's function and their derivatives stated in Section
\ref{secGreen} remain valid.

\section{Path regularity}\label{path}
In this section, we investigate the path regularity for the
stochastic solution of \eqref{S-CH-AC} under certain regularity
assumptions for the initial condition $u_0$ and when the domain
${ \mathcal{D}}$ is a parallelepiped. Note that the path regularity results proved in this section
remain valid on a rectangular domain for the equation \eqref{S-CH-ACn}. 
More precisely, we prove that when the coefficient $\sigma$ has
an appropriate sub-linear growth, 
 the  paths of the solution to  equation
\eqref{S-CH-AC} have a.s. a H\"older regularity depending on that of the initial condition. 
The path regularity proven here is the same as that obtained for  
the stochastic Cahn-Hilliard equation obtained in \cite{CW}, where
the coefficient $\sigma$ was supposed to be bounded. We follow
the main lines of the proof presented in \cite{CW}; nevertheless
some modifications are needed. 
 Indeed, the factorization method is used both for the deterministic and stochastic integrals. 

In this section we suppose that the assumptions of Theorem \ref{thexistence}  are satisfied.
 Let us recall that the integral form of the solution $u$ given by
\eqref{intf} 
can be decomposed as follows:
\begin{equation}\label{decompose-u}
u(t,x)=G_t u_0(x) + {\mathcal I}(x,t) + {\mathcal L}(x,t),
\end{equation}
where $G_t u_0$ is defined by \eqref{Gt}, and
\begin{align}
{\mathcal I}(x,t) = &\int_0^t \int_{{\mathcal{D}}}[\Delta G(x,y,t-s) - G(x,y,t-s)]  f(u(y,s)) \;dyds, \nonumber \\
{\mathcal L}(x,t)=
&\int_0^t\int_{{\mathcal{D}}}G(x,y,t-s)\sigma(u(y,s))\;
W(dy, ds).  \label{defL}
\end{align}
Let us study the regularity  of each term in the  decomposition \eqref{decompose-u} of $u$.

The  series decomposition of $G$ given in \eqref{Green} is similar to that in \cite{CW};
 hence an argument similar to
the proof of Lemma 2.1 of \cite{CW} if $u_0$ is continuous, and
to the first part of Lemma 2.2 of \cite{CW} if $u_0$ is
$\delta$-H\"{o}lder continuous, yields the following regularity
result for $G_\cdot u_0(\cdot)$.
\begin{lemma}\label{lem-path}
If $u_0$ is continuous, then the function $G_tu_0$ is continuous.
If $u_0$ belongs to $C^\delta({\mathcal{D}})$ for $0<\delta<1,$
then the function $(x,t)\to G_tu_0(x)$ is $\delta$-H\"{o}lder continuous in
the space variable $x$ and $\frac{\delta}{4}$-H\"{o}lder continuous in the time variable $t$.
\end{lemma}
Let us now consider the drift term ${\mathcal I}(x,t)$ and use
the factorization method (see e.g. \cite{DaZa} or \cite{CW}).

We remark that, as proved in Theorem \ref{thexistence}, if $u_0$
is bounded, then $u$ belongs a.s. to $L^{\infty}( 0,T ;
L^q({\mathcal{D}}))$ for any $q<\infty$ large enough.

The definition of the  Green's function yields
\begin{equation}\label{GG1}
\Delta G(x,y,t)=\int_{{\mathcal{D}}}G(x,y,t-s)\Delta G(z,y,s)\;
dz,
\end{equation}
and
\begin{equation}\label{GG2}
G(x,y,t)=\int_{{\mathcal{D}}}G(x,y,t-s) G(z,y,s)\; dz.
\end{equation}
For some $a\in(0,1)$ define the operators  ${\mathcal F}$ and
${\mathcal H}$ on $L^\infty(0,T ;L^q({\mathcal{D}}))$ as follows:
\begin{align*}
\mathcal{F}(v)(t,x):&=\int_0^t\int_{{\mathcal{D}}}G(x,z,t-s)(t-s)^{-a}v(z,s)\;
dzds,\\
\mathcal{H}(v)(z,s):&=\int_0^s\int_{{\mathcal{D}}} \big[ \Delta
G(z,y,s-s') - G(z,y,s-s')\big]  (s-s')^{a-1}  f(v(y,s'))\; dyds'   .
\end{align*}
Therefore, using relations \eqref{GG1} and \eqref{GG2} we deduce that 
\[ {\mathcal I}(x,t)=c_a\mathcal{F}(\mathcal{H}(u))(x,t),\]
where $c_a:=\pi^{-1} \sin(\pi a)$ obviously depends only on $a$.

First we claim that,  for $q$ satisfying condition ($\tilde{\mbox{\bf C}_\alpha}$),  the operator $\mathcal{H}$
maps  $L^\infty(0,T ;L^q({\mathcal{D}}))$ into itself. Indeed,
the estimates on the Green's  function in Lemma \ref{G-est-lemma}
and arguments similar to those used in Section \ref{sectruncated}
to prove \eqref{m_1-1-5} with $\rho = \frac{q}{3}$ (based on the
Minkowski and Young inequalities) prove that if $v\in
L^\infty(0,T ;  L^q({\mathcal{D}}))$ then
\begin{align*}
\| {\mathcal H}(v)(\cdot, t)\|_{q} & \leq \int_0^t
(t-s)^{-1+a-\frac{1}{2}-\frac{d}{2q}} \left( 1+\|v(\cdot, s)\|_q^3
\right)\; ds.
\end{align*}
For the boundedness of the above integral we need that
$-1+a-\frac{1}{2}-\frac{d}{2q}> -1$. Since $q>d$, this inequality holds for some 
$a\in \big(\frac{1}{2}+\frac{d}{2q},1\big)$. 
Then, an argument similar to that used in
\cite{CW} proves that if $v\in L^\infty ([0,T] ;
L^q({\mathcal{D}}))$ then ${\mathcal J}(v)$ belongs to ${\mathcal
C}^{\lambda , \mu}({\mathcal{D}}\times [0,T])$ for any $\lambda<1$
and $\mu<\frac12$. Indeed, the upper estimates of the Green's
function from Lemma \ref{Green} are the same as that for the
Green's function of the Cahn-Hilliard equation which only involves
the fourth order derivatives.

Considering the stochastic integral ${\mathcal L}$ defined in
\eqref{defL}, we observe that the fact that $\sigma$ is not
bounded any more does not allow us to use the related argument
from the proof of Lemma 2.2 in \cite{CW} stated on page 797.
Instead, we also  use  the factorization method for the stochastic integral. 
 Recall that given any $n\geq 1$, for $T_n=\inf\{ t \geq 0 : \|u_n(.,t)\|_q \geq n\}$ we have $1_{\{ T\leq T_n\}}
u(.,t)=1_{\{ T\leq T_n\}}u_n(.,t)$, where $u_n$ is the solution to
\eqref{u_n}. The local property of stochastic integrals implies
that for any $n$ and $t\in [0,T]$:
\[  1_{\{ T\leq T_n\} } {\mathcal L}(x,t) =  1_{\{ T\leq T_n\} }
 \int_0^t \int_{\mathcal{D}} G(t-s,x,y) 1_{\{ s\leq T_n\} }\sigma(u_n(y,s)) W(dy,ds)  . \]
The process $u_n$ is adapted and by \eqref{sup_u_n} for $q$ large enough,
if Condition (C) holds true, $\gamma >0 $ and  $\beta \in (1,\infty)$ are such that $\beta \gamma \in [2,\infty)$,
we have
 $E\big| \int_0^T \|u_n(.,t)\|_q^\gamma dt|^\beta  \leq C(n,T)$.
 Fix   $a\in (0,1)$, let ${\mathcal K}(u_n)$
be defined as follows:
\[ {\mathcal K}(u_n)(x,t)=\int_0^t \int_{\mathcal{D}} G(x,y,t-s) (t-s)^{a-1} 1_{ \{ s\leq T_n\} } \sigma (u_n(y,s)) W(dy,ds).\]
We at first check that this stochastic integral makes sense for fixed $t\in [0,T]$ and $x\in {\mathcal D}$,
and that a.s. ${\mathcal K}(u_n)\in L^\infty(0,T ; L^q({\mathcal D}))$,
 so that $1_{\{T\leq T_n\}} {\mathcal L}(x,t) = 1_{\{T\leq T_n\}} c_a
{\mathcal F}({\mathcal K}(u_n))(x,t)$.

Indeed, for fixed $t\in [0,T]$, $x\in {\mathcal D}$ and $p\in [1,\infty)$, the Burkholder inequality yields 
\begin{align*}
E|{\mathcal K}(u_n)(x,t)|^{2p } \leq & 
E \Big| \int_0^t \int_{\mathcal D} G^2(x,y,t-s) (t-s)^{2(a-1)} 1_{\{ s\leq T_n\}} \sigma^2(u_n(s,y)) dy ds\Big|^p\\
\leq & C(n) \Big| \int_0^t
(t-s)^{-\frac{d}{2}+2(a-1)+\frac{d}{4}}ds\Big|^p
\end{align*}
Let  $a\in(\frac{1}{2}+ \frac{d}{8},1)$; then we have 
$-\frac{d}{2}+2(a-1)+\frac{d}{4}>-1$, which yields 
\begin{equation}\label{neq1}
E|{\mathcal K}(u_n)(x,t)|^{2p} <\infty,\;\;\forall p \in [1,\infty).
\end{equation}

Let us now prove moment upper estimates of increments of
${\mathcal K}(u_n)$; this together with \eqref{neq1} will imply
by Garsia's Lemma that
$$E\big( \|{\mathcal K}(u_n)\|_{L^\infty(\mathcal{D}\times  [0,T] } )^{2\rho} \big) <\infty.$$

Arguments similar to those used in the proof of \eqref{supx} prove
that for $\tilde{\lambda} \in (0,1)$, $\tilde{q}\in (2\alpha, q)$
and $n\geq 1$,  we have for $t\in [0,T]$, $x,\xi \in {\mathcal D}$:
\begin{align*}
E\big| &{\mathcal K}(u_n)(x,t) - {\mathcal K}(u_n)(\xi,t)\big|^{2p} \leq C_p |x-\xi|^{2\tilde{\lambda} p} \\
&\quad \times \Big| \int_0^t (t-s)^{-\frac{d+1}{2}\tilde{\lambda} - \frac{d}{2}(1-\tilde{\lambda}) + 2(a-1)}
 1_{\{s\leq T_n\}} \| \exp(h(.,t,s)\|_{\frac{\tilde{q}}{\tilde{q}-2\alpha}}
\big[ 1+\|u_n(.,s)\|_{\tilde{q}}^{2\alpha}\big] ds \Big|^p\\
&\leq C_p(n) |x-\xi|^{2\tilde{\lambda}p} \Big| \int_0^t (t-s)^{-\frac{d+\tilde{\lambda}}{2}
+ 2(a-1) + \frac{d(\tilde{q}-2\alpha)}{4\tilde{q}}}
ds\Big|^p\\
&\leq C_p(n,T) |x-\xi|^{2\tilde{\lambda}p}
\end{align*}
for some finite constant $C_p(n,T)$, provided that the time integrability constraint 
$-\frac{d+\tilde{\lambda}}{2} + 2(a-1) +
\frac{d(\tilde{q}-2\alpha)}{4\tilde{q}}>-1$ holds true.    Since  Condition
($\tilde{\mbox{\bf C}_\alpha}$) is satisfied with $\alpha \in (0,\frac{1}{9})$, we deduce that $q>\frac{2\alpha d}{4-d}$. 
Hence given $\bar{\lambda} \in \big( 0, 2-\frac{d}{2}\big)$ one can find 
$\tilde{q}\in \big( \frac{2\alpha d}{4-d}, q \big)$ and $a\in \big( \frac{1}{2}+\frac{d}{8},1\big)$  such that the time
integrability is fulfilled. 

Similarly, for $0\leq t'\leq t \leq T$, $x\in {\mathcal D}$ and $\tilde{\mu}\in \big(0,\frac{1}{2}-\frac{d}{8}\big)$, 
arguments similar to that proving \eqref{supt}
imply
\begin{align*}
E\big| &{\mathcal K}(u_n)(x,t) - {\mathcal K}(u_n)(x,t')\big|^{2p} \leq C_p |t-t'|^{2\tilde{\mu} p} \\
&\quad \times \Big| \int_0^t (t-s)^{-2 (\frac{d}{4}+1)\tilde{\mu} - \frac{d}{2}(1-\tilde{\mu}) + 2(a-1)}
 1_{\{s\leq T_n\}} \| \exp(h(.,t,s)\|_{\frac{\tilde{q}}{\tilde{q}-2\alpha}}
\big[ 1+\|u_n(.,s)\|_{\tilde{q}}^{2\alpha}\big] ds \Big|^p\\
&\leq C_p(n) |t-t'|^{2\tilde{\mu}p} \Big| \int_0^t (t-s)^{-\frac{d}{2} -2\tilde{\mu}  + 2(a-1) + \frac{d(\tilde{q}-2\alpha)}{4 \tilde{q}}}
ds\Big|^p\\
&\leq C_p(n,T) |t-t'|^{2\tilde{\mu}p}
\end{align*}
for some finite constant $C_p(n,T)$, provided that
$2a-2\tilde{\mu} > 1+\frac{d}{4} + \frac{\alpha d}{2 \tilde{q}}$.
Once more, since $\alpha \in \big( 0, \frac{1}{9} \big)$, we have $q>\frac{2\alpha d}{4-d}$ and 
given $\tilde{\mu}\in \big( 0, \frac{1}{2}-\frac{d}{8}\big)$, we can find 
$\tilde{q}\in \big( \frac{2\alpha d}{4-d}, q \big)$ such that this inequality holds true.

 Hence, given $\bar{\lambda} \in \big( 0, 2-\frac{d}{2} \big)$ and $\bar{\mu} \in \big( 0, \frac{1}{2}-\frac{d}{8}\big) $, 
for every $n\geq 1$ and $p\in [1,\infty)$,   we can find some positive constant $C_p(n,T)$ such that 
 \[ E|{\mathcal K}(u_n)(x,t)-{\mathcal K}(u_n)(\xi,t')|^{2p} \leq C_p(n,T) \big( |\xi-x|^{2\tilde{\lambda} p} + |t-t'|^{2\tilde{\mu}p}\big)\]
for $0\leq t'\leq t \leq T$ and $x,\xi\in {\mathcal D}$.

The Garsia-Rodemich-Rumsey lemma implies that  
$$E\big(  \|{\mathcal K}(u_n)\|_{L^\infty(\mathcal{D}\times [0,T])}^{2p }\big) <\infty , \;\;\;\forall p \geq 1, $$ 
and 
$$E\big(  \|{\mathcal  K}(u_n)\|_q^{2p }\big) \leq E\big(  \|{\mathcal
K}(u_n)\|_{L^\infty(\mathcal{D} \times [0,T] )}^{2p} \big) <\infty , \;\;\;\forall p \geq 1. $$
 This  gives one one hand the stated time and space H\"older regularity, and on the other
 hand  the previous 
space-time H\"older moments estimates of 
 $ {\mathcal K}(u_n)\in L^\infty(0,T, L^q({\mathcal D}))$ a.s.

Since ${\mathcal F}$
maps  $L^\infty(0,T ; L^q({\mathcal D}))$ into ${\mathcal C}^{\lambda,\mu}( {\mathcal D}\times [0,T])$ for
 $\lambda < \frac{1}{2}-\frac{d}{8}$ and
$\mu < 2- \frac{d}{2}$ and since
${\mathcal L}(x,t) =  c_a {\mathcal F}({\mathcal K}(u_n))(x,t)$ on the set $\{ T\leq T_n\}$, we deduce that
${\mathcal L}(u)\in {\mathcal C}^{\lambda,\mu}({\mathcal D}\times [0,T])$ a.s. on the set $\{ T\leq T_n\}$.

 Finally, Theorem \ref{thexistence} implies that
as $n\to \infty$ the sets
$\{T\leq T_n\} $ increase to $\Omega$; this proves that a.s. ${\mathcal L}(u)\in {\mathcal C}^{\lambda,\mu}
([0,T]\times {\mathcal D})$
for $\lambda < \frac{1}{2}-\frac{d}{8}$ and $\mu < 2-\frac{d}{2}$.

As a consequence (cf. \cite{CW}), we obtain the following regularity of the trajectories.
\begin{theorem}\label{pathreg}
  Let  $\sigma$ be Lipschitz and satisfy the sub linearity condition
 \eqref{undif} with  $\alpha \in (0, \frac{1}{9})$,  let   $q$ satisfy Condition ($\tilde{\mbox{\bf C}_\alpha}$), 
 and let $u_0\in L^q({\mathcal D})$.  Then

(i) If $u_0$ is continuous, then the solution of \eqref{intf} has
almost surely continuous trajectories.

(ii)  If $u_0$ is
$\beta$-H\"{o}lder continuous for $0<\beta<1$, then the
trajectories of the solution to  \eqref{intf} are almost surely
$\beta\wedge \big(2-\frac{d}{2}\big)$-continuous in space and
$\frac{\beta}{4}\wedge \big(\frac{1}{2}-\frac{d}{8}\big)  $-continuous in time.  
\end{theorem}


\noindent {\bf Acknowledgments} Dimitra Antonopoulou and Georgia Karali are
supported by the ``ARISTEIA" Action of the ``Operational Program
Education and Lifelong Learning" co-funded by the European Social
Fund (ESF) and National Resources.


\end{document}